\documentclass[reqno,11pt,a4paper]{article}

\usepackage[T1]{fontenc}
\usepackage[utf8]{inputenc}
\usepackage[english]{babel}
\usepackage[height=25cm, width=16cm, bindingoffset=1cm]{geometry} 

\usepackage[intlimits]{amsmath}
\usepackage{lmodern}
\usepackage{latexsym} 			
\usepackage{amsfonts} 			
\usepackage{amssymb} 			
\usepackage{stmaryrd}			
\usepackage{mathtools} 			
\usepackage{physics}			
\usepackage[bbsets, subsetnonamb]{jkmath}
\usepackage{yhmath}				
\usepackage[greeklower, frac, autobold, root, rootclose]{mathfixs}
\usepackage{mathdots} 			

\usepackage{bm}
\usepackage[autostyle]{csquotes}
\usepackage{todonotes}			
\usepackage[final]{hyperref} 	
\usepackage{footnotebackref} 	
\usepackage{microtype}			
\usepackage{paralist}			
\usepackage{stackrel}			
\usepackage[english, showdow]{datetime2} 

\usepackage{pgf,tikz,pgfplots}
\usepackage[thmmarks, amsmath, hyperref, framed]{ntheorem}			
\usepackage{autonum}			
\usepackage{xcolor}
\usepackage[section]{placeins}	
\usepackage{xspace}				

\usepackage{subcaption}			
\usepackage{float}

\newtheoremstyle{Stil}{\item[ \textbf{##1 ##2. }]}{\item[\textbf{##1 ##2 (##3).}]}
\theorembodyfont{\rmfamily}

\theoremstyle{Stil}

\newtheorem{Satz}{Theorem}[section]

\newtheorem{Bem}[Satz]{Remark}
\newtheorem{Thm}[Satz]{Theorem}
\newtheorem{Lem}[Satz]{Lemma}

\newtheorem{Kor}[Satz]{Corollary}
\newtheorem{Assump}[Satz]{Assumption}

\newtheorem{Example}[Satz]{Example}

\newtheoremstyle{StilB}{\item[ \textit{##1: }]}{\item[ \textit{##3:}]}
\theoremstyle{StilB}

\theoremsymbol{\ensuremath{\square}}
\newtheorem{proof}{Proof}

\providecommand{\keywords}[1]{\textbf{\textit{Key words---}} #1}

\providecommand{\MSC}[1]{\textbf{\textit{MSC---}} #1}

\makeatletter
\newcommand*{\centerfloat}{
	\parindent \z@
	\leftskip \z@ \@plus 1fil \@minus \textwidth
	\rightskip\leftskip
	\parfillskip \z@skip}

\AtBeginDocument{
	\expandafter\renewcommand\expandafter\subsection\expandafter{%
		\expandafter\@fb@secFB\subsection
	}%
}

\makeatother

\DeclareRobustCommand{\eins}{%
	\text{\usefont{U}{dsrom}{m}{n}1}%
}

\newlength{\leftstackrelawd} 
\newlength{\leftstackrelbwd}
\def\leftstackrel#1#2{\settowidth{\leftstackrelawd}%
	{${{}^{#1}}$}\settowidth{\leftstackrelbwd}{$#2$}%
	\addtolength{\leftstackrelawd}{-\leftstackrelbwd}%
	\leavevmode\ifthenelse{\lengthtest{\leftstackrelawd>0pt}}%
	{\kern-.5\leftstackrelawd}{}\mathrel{\mathop{#2}\limits^{#1}}}

\newcommand{\hilbert}{\ensuremath{\mathcal{H}}}

\newcommand{\laplace}{\Delta}

\newcommand*{\tran}{^{\mkern-1.5mu\mathsf{T}}}

\newcommand{\leqc}{\lesssim}

\let\temp\phi
\let\phi\varphi
\let\varphi\temp
\let\temp\epsilon
\let\epsilon\varepsilon
\let\varepsilon\temp
\let\temp\rho
\let\rho\varrho
\let\varrho\temp
\let\temp\theta
\let\theta\vartheta
\let\vartheta\temp

\let\temp\Re
\let\Re\real
\let\real\temp
\let\temp\Im
\let\Im\imaginary
\let\imaginary\temp

\DeclarePairedDelimiterX{\skpu}[1]{\langle}{\rangle}{#1}
\newcommand{\skp}[1]{\skpu*{#1}} 
\DeclarePairedDelimiterX{\skpru}[1]{(}{)}{#1}
\newcommand{\skpr}[1]{\skpru*{#1}} 

\DeclareMathOperator{\chak}{char}

\newcommand{\amax}{{a_{\text{max}}}}

\begin{document}
	\title{Optimal Control of General Nonlocal Epidemic Models with Age and Space Structure}
	\author{Behzad Azmi\thanks{Department of Mathematics and Statistics, University of Konstanz, Universit\"atsstra\ss e 10, 78457 Konstanz, Germany. E-Mail: behzad.azmi@uni-konstanz.de}\and Nicolas Schlosser\thanks{Department of Mathematics and Statistics, University of Konstanz, Universit\"atsstra\ss e 10, 78457 Konstanz, Germany. E-Mail: nicolas.schlosser@uni-konstanz.de}}
	\date{\today}
	\maketitle

\begin{abstract}

We analyze a class of general nonlinear epidemic models with age and spatial structures, including a nonlocal infection term dependent on age and space. After establishing the well-posedness of the state partial differential equation, we introduce a control parameter representing the vaccination rate. Under suitable conditions, we prove the existence of an optimal control and characterize it through first-order necessary optimality conditions. Finally, we present numerical examples to illustrate the behavior of the optimal control strategies governed by these models.

\end{abstract}


\MSC{35Q92 \and 92D30 \and 49M41 \and 35K57  \and 49N90}

\section{Introduction}
Mathematical epidemiology is over a hundred years old, and the recent COVID-19 pandemic has demonstrated its significant impact on everyone's lives. Over time, numerous models have been developed, each surpassing the previous one in complexity. Notably, we can mention the first model by Kermack and McKendrick \cite{KermackMcKendrick}, which consists of the following ordinary differential equations
\begin{equation}
	\dot{S} = - \lambda I S, \quad \dot{I} = \lambda I S - \gamma I, \quad \dot{R} = \gamma I
\end{equation} 
where $S$, $I$, and $R$ stand for susceptible, infectious and removed individuals respectively. While this early model already exhibits interesting behavior and provides valuable insights into the general progression of epidemics, it is far too simplistic for real-world applications. Consequently, later models build upon the framework of the \emph{SIR} model by introducing additional compartments and variables. More complex models incorporate an age variable to account for variations in infection or mortality rates across different age groups, as well as the potential impact of an epidemic on fertility. Many models also include variables for physical space, enabling the simulation of individual movement and the geographic spread of a pandemic across a country or region. Spatial movement is typically modeled using a diffusion term that includes the Laplace operator acting on the spatial variables. Meanwhile, the aging process introduces a first-order derivative term, making it more challenging to classify the resulting model as purely parabolic or hyperbolic. For further details, see, for example, \cite{Walker} or \cite{Webb}.

This manuscript focuses on a general model class that incorporates both age and spatial structure. When considering age- and space-structured models, obtaining realistic results necessitates accounting for infection processes that are generally nonlocal. This means that every susceptible individual can become infected by any infectious individual, not just those in the same location and of the same age. Consequently, the $\lambda I S$ terms in the above models must be rewritten as $\Lambda(I)S$, where
\begin{equation}
	\Lambda(a, x, I) = \int_0^\amax \int_\Omega k(a, \alpha, x, \xi) I(\alpha, \xi) \dd{\xi} \dd{\alpha},
\end{equation}
 with the maximal age $\amax$ of the population, $\Omega$ the domain in space in which the population lives, and the kernel $k$ describing the infectiousness of individuals at age $\alpha$ and location $\xi$ to those at age $a$ and location $x$. Many other processes, such as noncompliance (as discussed in \cite{Bongarti}), can also be represented in this form. From a practical perspective,  it is of great interest to impose a control parameter, such as a vaccination rate, on models that include this kind of terms. This is the objective of our article. Instead of focusing on a single model, we explore the general form of epidemic models with nonlocal terms, along with age and space structures. We discuss the potential outcomes and benefits of this approach.

In many cases, it is not only desirable to predict the course of an epidemic but also to actively influence the spread of the disease. Vaccines are an important tool in this effort, as they can lower the risk of infection and reduce the severity of symptoms (and, in some cases, infectiousness). Vaccinations can be easily incorporated into epidemic models, typically in the form $u(t, a, x) S(t, a, x)$ where $u$ 
is the vaccination rate.

One obvious question is what the optimal vaccination strategy should be. This includes considerations such as which age groups to prioritize, when to administer doses, at what stage of the epidemic, and how to balance the cost of vaccination with the number of cases. Questions like these can be answered using methods from optimal control theory by optimizing the cost functional
 \begin{equation}
 \label{eq:Ob_function}
		J(u, I) = \frac 12 \int_0^T \int_0^\amax \int_\Omega \abs{I(t, a, x)}^2 \dd{x} \dd{a} \dd{t} + \frac \alpha2 \int_0^T \int_0^\amax \int_\Omega \abs{u(t, a, x)}^2 \dd{x} \dd{a} \dd{t}
	\end{equation}
    for some given weight $\alpha \geq 0$ subject to the partial differential equation epidemic model. The first term in $J$ models the number of infected individuals, and the second term models the cost of the vaccine.  Typically, control constraints
    \begin{equation}
    \label{eq:Control_C}
        0 \leq u(t,\alpha,x) \leq \bar{u}   \quad  \text{ for a.e. }  \quad  (t,\alpha,x) \in (0,T)\times (0,a_{\max}) \times \Omega 
        \end{equation}
   are imposed
with some maximal vaccination capacity $\bar{u} > 0$. In real-world applications, the vaccine can only be administered in certain locations or uniformly for certain age groups. In this paper, we model all of these effects.

\subsection{The State Model}

Let $\Omega \subset \R^2$ be a spatial domain with sufficiently smooth boundary and outer normal $\nu$, and let $0 < \amax < \infty$ the maximal age of the population. We consider general models of the form 
\begin{equation}
\label{eq:StateEquation}
	\begin{split}
	    \delta y + L(a, x) y + \Lambda(a, x, y) y + K(u) y &= \sigma(a) \laplace y,\\
	y(t = 0) = y_0, \quad y(a = 0) = \int_0^\amax \beta(\alpha, x) y(t, \alpha, x) \dd{\alpha} &\eqqcolon B,\\
	\partial_\nu y(x \in \partial \Omega) &= 0
	\end{split}
\end{equation}
where $y$ is a vector of compartments, e.g. $y:= (S, I, R)$. The letter $\delta$ stands for $\partial_t + \partial_a$ in the weak sense, and $\laplace$ is the usual Laplacian in the space variables. In the operator $L$ we collect all linear reaction terms, and all nonlocal quasilinear ones in $\Lambda(y)$. Finally, $K(u)$ is a linear control operator depending linearly on the control parameter $u$, and $\sigma$ is a diagonal matrix of diffusion coefficients. The infection operator $\Lambda(a, x, y)$ is given by (we use Einstein's notation; $h$, $i$, $j$ are indices of the 3-tensor $k$) \begin{equation}
	\Lambda(a, x, y)^{hi} = \int_0^\amax \int_\Omega k^{hij}(a, \alpha, x, \xi) y_j(\alpha, \xi) \dd{\xi} \dd{\alpha} = \skpr{k(a, \cdot, x, \cdot), I}_{L^2((0, \amax) \times \Omega)},
\end{equation}
where $k^{hij}$ describes how many susceptibles $y_i$ are  
infected from infectives $y_j$ and subsequently transition into the infected class $y_h$. This term allows for the possibility of multiple infected compartments, as seen in many models of COVID-19 which include asymptomatic infectives or superspreaders.

In a particular case of model \eqref{eq:StateEquation}, we will consider the following \textit{SVIR} model with loss of immunity, adapted from \cite{AbRaSch},
\begin{equation}
\tag{SVIR}
\label{eq:svir}
    \begin{split}
    \delta S - \sigma_S \Delta S &= c V - \left[u + \mu + \Lambda(I) \right] S,\\
	\delta V - \sigma_V \Delta V &= u S - \left[\mu + c + \phi_1 \Lambda(I) \right] V,\\
	\delta I - \sigma_I \Delta I &= \Lambda(I) (S + \phi_1 V + \phi_2 R) - (\mu + \delta + \gamma) I,\\
	\delta R - \sigma_R \Delta R &= \gamma I - \left[\mu + \phi_2 \Lambda(I) \right] R
\end{split}
\end{equation}
 as a special case, where $u$ is the vaccination rate, $c$ the rate of loss of vaccine protection, $\mu$ the natural death rate, $\phi_1$ the vaccine efficacy, $\phi_2$ the protection due to immunity and $\delta$ the infection-induced death rate. We assume that infection does not affect fertility and that all newborns are susceptible. Hence we are given the conditions \begin{gather}
     S(a = 0) = \int_0^\amax \tilde{\beta}(\alpha) (S + V + I + R)(\alpha) \dd{\alpha}, \quad V(a = 0) = I(a = 0) = R(a = 0) = 0,
 \end{gather} where $\tilde{\beta}$ is the birth rate. All parameters except $\sigma$ may depend on age and space ($\sigma$ is only allowed to depend on age), $u$ is our control parameter, which also may depend on time. The model \eqref{eq:svir} together with the boundary conditions for $a = 0$ can be written in the form of \eqref{eq:StateEquation} with \begin{gather}
	y = \begin{psmallmatrix}
		S\\V\\I\\R
	\end{psmallmatrix}, \quad \sigma = \begin{psmallmatrix}
		\sigma_S &&&\\
		&\sigma_V &&\\
		&&\sigma_I &\\
		&&&\sigma_R
	\end{psmallmatrix}, \quad L = -\begin{psmallmatrix}
		-\mu & c &&\\
		& -(\mu + c) &&\\
		&& - (\mu + \delta + \gamma) &\\
		&& \gamma & -\mu
	\end{psmallmatrix}, \quad K(u) = -u \cdot \begin{psmallmatrix}
		-1 &&&\\
		1 &0&&\\
		&&0&\\
		&&&0
	\end{psmallmatrix},\\
	\Lambda(a, x, y) = -\skpr{\lambda(a, x, \cdot, \cdot), y_3}_{L^2((0, \amax) \times \Omega)} \cdot \begin{psmallmatrix}
		-1 &&&\\
		&-\phi_1 &&\\
		1 & \phi_1 &0 & \phi_2\\
		&&&-\phi_2
	\end{psmallmatrix}, \quad \beta = \tilde{\beta} \cdot \begin{psmallmatrix}
		1 & 1 & 1 & 1\\
		&0&&\\
		&&0&\\
		&&&0
	\end{psmallmatrix},
\end{gather}
which amounts to \begin{equation}
	k^{hij} = \begin{cases}
		\lambda \cdot\begin{psmallmatrix}
			-1 &&&\\
			&-\phi_1 &&\\
			1 & \phi_1 &0 & \phi_2\\
			&&&-\phi_2
		\end{psmallmatrix}, & j = 3,\\
		0, & j \neq 3.
	\end{cases}
\end{equation}

\subsection{Related Works}

The first epidemic models were introduced in \cite{KermackMcKendrick} and have since grown significantly in complexity. Comprehensive introductions to the subject can be found in the monographs \cite{BDJ,LiYaMa,Perthame,Webb} and the references therein. The study \cite{WangZhangKuniya} examines a space-structured model with Dirichlet boundary conditions. The works \cite{Colombo} and \cite{Kuniya} explore age- and space-structured models with nonlinear infection terms. The former employs first-order transport-equation-like terms, while the latter uses diffusion terms; however, neither work investigates the well-posedness of the state equation.

In \cite{Walker2}, a model incorporating both spatial structure and infection age is studied, where only the infected population depends on the age variable. \cite{BDKW} and \cite{KangRuan} address age-structured models that include nonlocal diffusion terms to account for long-distance travel. \cite{Fragnelli} considers linear age- and space-dependent population models where an additional delay term is present in the birth equation. Meanwhile, \cite{BredaReggiVermiglio} focuses on the numerical aspects of such models.

There is also extensive literature on optimal control in mathematical epidemiology. Notable among these are the works \cite{SharomiMalik} and \cite{YusufBenyah}, which discuss the optimal control of epidemic models formulated using ordinary differential equations. The book \cite{AnitaArnautuCapasso} provides numerous examples of optimal control problems, including results specific to age-structured population models. Additionally, the articles \cite{AbRaSch,ACGRR,Bongarti,CGMR,ZhouXiangLi} investigate optimal control problems in reaction-diffusion epidemic models that incorporate spatial structure but not age structure. To the best of our knowledge, the optimal control of models incorporating both age and spatial structures has not been addressed in previous research.

\subsection{Organization of the Paper}

The remainder of the paper is organized as follows: Section 2 investigates the well-posedness of the state equation using a fixed-point argument and establishes the continuity of its solution with respect to the forcing functions and initial data. Building on these findings, Section 3 addresses the well-posedness of the optimal control problems. Section 4 derives and presents the first-order optimality conditions for these problems. Section 5 provides numerical experiments that illustrate and discuss the qualitative behavior of the model. Finally, concluding remarks are presented.

\subsection{Notation}
We write $\mathbb{R}_{\geq 0}$ to denote the set of non-negative real numbers. For any set $\Omega$ and any subset $\Omega_0 \subset \Omega$ we denote by $\eins_{\Omega_0}:\Omega \to \set{0, 1}$ the indicator function which maps $\omega \in \Omega$ to $1$ if $\omega \in \Omega_0$ and to 0 otherwise. For a Banach space $X$, we denote the associated norm by $\|\cdot\|_X$, the dual space by $X'$, and the dual pairing between $X'$ and $X$ by $\skp{\cdot, \cdot}_{X',X}$. If $X$ is a Hilbert space, we use the scalar product $\skpr{\cdot, \cdot}_X$. Furthermore, $\mathcal{L}(X, Y)$ denotes the space of continuous linear operators from $X$ to $Y$, equipped with the usual operator norm $\|\cdot\|_{\mathcal{L}(X,Y)}$. For a smooth and bounded domain $\Omega \subset \R^d$ (in applications, typically $d$ equals two), we consider the spaces 
\begin{equation}
    H \coloneqq L^2(\Omega)^n, \quad V \coloneqq H^1(\Omega)^n,
\end{equation}
where $n$ is the number of compartments in our epidemic model. We define $V'$ to be the dual of $V$ via the dual pairing induced by the inner product on $H$. That is,  we have $V \hookrightarrow  H=H' \hookrightarrow  V'$, where the embeddings are dense and compact. For any open interval $(t_0,t_1) \subset \mathbb{R}_{\geq 0}$, we introduce the space
\begin{equation}    
W((t_0,t_1), V, V') \coloneqq \set{ y \in L^2((t_0,t_1),V)   \where   \partial_t y \in L^2((t_0,t_1),V')},
\end{equation}
equipped with the norm
\begin{equation}
\|y\|_{W((t_0,t_1),V,V') } = \left( \|y\|_{L^2((t_0,t_1),V) }^2  +\|\partial_t y\|_{L^2((t_0,t_1),V') }^2 \right)^{\frac{1}{2}}
\end{equation}
where the derivative $\partial_t $ is understood in the sense of distributions. It is well-known (see, e.g., \cite[Lem. 11.4]{RenardyRogers} or \cite[Thm. XVIII.1.1]{DautrayLions}) that the space $W((a, b), V, V')$ embeds continuously into $C([a, b], H)$.

The age variable is assumed to be bounded by a finite maximal age, denoted by  $\amax > 0$. For brevity, we define $\mathcal{I} \coloneqq (0, \amax)$. We will frequently work in the space $\hilbert \coloneqq L^2(\mathcal{I}, H) = L^2(\mathcal{I} \times \Omega)^n$. 
 To facilitate our analysis, we fix a time horizon $0 < T < \infty$. It will turn out to be very helpful to partition the variable space $[0, T] \times \overline{\mathcal{I}}$ into sets of the form
\begin{equation}
    \chak(t_0) \coloneqq \set{(t_0 + h, h) \where 0 \leq h \leq \amax} \cap ([0, T] \times \overline{\mathcal{I}})
\end{equation}
where $t_0 \in [-\amax, T]$. These sets represent the so-called \emph{characteristic lines} and are illustrated in Figure \ref{fig:CharakteristikTkA}. We frequently restrict functions $\phi$ defined on $[0, T] \times \bar{\mathcal{I}}$ on these characteristics, for which we will use the notation $\restr{\phi}_{\chak(t_0)}(h) \coloneqq \phi(t_0 + h, h)$, for all parameters $h \in [\max\set{-t_0, 0}, \min\set{T-t_0, \amax}]$.

Throughout the manuscript, within the context of a priori estimates, the notation \enquote{$A \leqc B$} denotes an inequality of the form \enquote{$A \leq c B$}, where $c$ is a generic constant independent of the quantities to be estimated.

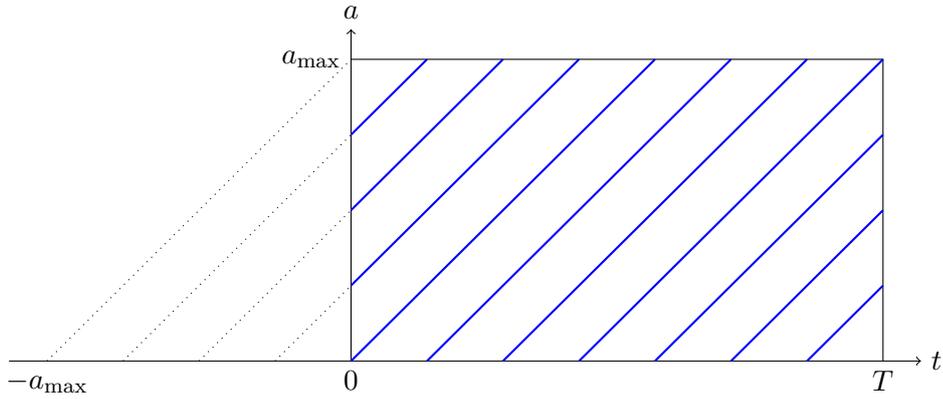
\begin{figure}
    \centering
    \begin{tikzpicture}
        \draw [->] (0,0) node[below]{0} -- (0,4.4) node[above]{$a$};
        \draw [->] (-4.5,0) -- (7.5,0) node[right]{$t$};
        \draw (0,4) node[left]{$\amax$} -- (7,4) -- (7,0) node[below]{$T$};
    	\draw [dotted] (-4,0) node[below]{$-\amax$} -- (0,4);
        \draw [dotted] (-3,0) -- (0,3); \draw [blue, thick] (0,3) -- (1,4);
        \draw [dotted] (-2,0) -- (0,2); \draw [blue, thick] (0,2) -- (2,4);
        \draw [dotted] (-1,0) -- (0,1); \draw [blue, thick] (0,1) -- (3,4);
        \draw [blue, thick] (0,0) -- (4,4);
        \draw [blue, thick] (1,0) -- (5,4);
        \draw [blue, thick] (2,0) -- (6,4);
        \draw [blue, thick] (3,0) -- (7,4);
        \draw [blue, thick] (4,0) -- (7,3);
        \draw [blue, thick] (5,0) -- (7,2);
        \draw [blue, thick] (6,0) -- (7,1);
	\end{tikzpicture}
	\caption{Schematic image of the characteristics (blue lines), a so-called \emph{Lexis diagram}. For every $t_0 \in [-\amax, T]$ we obtain a characteristic: If $t_0 \geq 0$, it starts in $(t, a) = (t_0, 0)$, for $t_0 < 0$ it starts in $(t, a) = (0, -t_0)$.}
    \label{fig:CharakteristikTkA}
\end{figure}

\section{Well-Posedness of the State Equation}

In this section, we demonstrate the existence and uniqueness of a solution to the state equation. This process is carried out in several steps. First, we analyze a linearized form of the equation along characteristic lines (which can be interpreted as following a specific age cohort), leading to a standard parabolic differential equation that can be addressed using well-established techniques. Next, we reconstruct the solutions across the various characteristic lines, incorporate the implicit birth law, and, in the final step, address the nonlinear equation using fixed-point arguments.

\subsection{The Linearized Equation with Fixed Birth Number}
For any fixed time horizon $0 < T < \infty$,  we consider the linearized equation where all $k^{\beta \gamma \delta} = 0$. We also ignore the control term at first, and assume that instead of the implicit birth law from  \eqref{eq:StateEquation}, we are given a fixed number $B$ of newborns.  Then for given $f$ and $y_0$, the equation takes the following form 
 \begin{equation}\label{eq:linear_with_age}
		\begin{split}
		    \delta y + L(a, x) y  - \sigma(a) \laplace y &= f,\\*
		y(t = 0) = y_0, \quad y(a = 0) &= B,\\*
		\partial_\nu y(x \in \partial \Omega) &= 0.
		\end{split}
\end{equation} 
 The following assumptions are made. 
\begin{Assump}
\label{assum:Lin}
Suppose that:
\begin{enumerate}
\item $B \in L^2((0, T), H)$, $f \in L^2((0, T) \times \mathcal{I}, V')$ and $y_0 \in L^2(\mathcal{I}, H)$.
\item $L \in C(\overline{\mathcal{I}}, L^\infty(\Omega))$ and its entries are uniformly bounded away from zero with respect to both age and space.
	\item $\sigma \in C(\overline{\mathcal{I}}, \R^{n \times n})$ is a diagonal matrix whose entries are uniformly bounded away from zero with respect to age.
\end{enumerate}
\end{Assump}

The linearized equation \eqref{eq:linear_with_age} is most effectively solved along characteristic lines, see e.g., \cite{Walker} and \cite[Sec. 1.3]{Webb}. By fixing a birth date $t_0 \in [-\amax, T]$ and setting  $v(h) = \restr{y}_{\chak(t_0)} = y(t_0 + h, h)$, we obtain the following system
\begin{equation} \label{eq:linChar}
	\begin{split}
	    v_h  + L(h, x) v  - \sigma(h) \laplace v &=  \restr{f}_{\chak(t_0)},\\*
	v(h = \max\{0, -t_0\}) &= \begin{cases}
		B(t_0, x) & t_0 > 0,\\
		y_0(-t_0, x) & t_0 < 0,
	\end{cases}\\*
	\partial_\nu v(x \in \partial \Omega) &= 0.
	\end{split}
\end{equation}
 Note that the structure of the equation is independent of $t_0$, only the initial conditions and inhomogeneous terms depend on it.

\begin{Lem}
\label{lem:well_linear}
For any given $f \in L^2((\max\set{0, -t_0}, \amax), V')$, \eqref{eq:linChar} admits a unique weak solution $v \in W((\max\set{0, -t_0}, \amax), V, V')$. 
\end{Lem}

\begin{proof}
	We observe that $\mathcal{A}(h) \coloneqq \sigma(h) \laplace - L(h, \cdot)$ is a bounded linear operator from $V$ to $V'$, which is continuous with respect to $h$. From the boundedness away from zero, we can deduce (weak) coercivity, i.e. $-\skp{\mathcal{A}u, u}_{V' \times V} \geq a \norm{u}_V^2 - b \norm{u}_H^2$ for some $a$, $b > 0$ independent of $h$. Then the result follows from standard results in parabolic theory, see e.g., \cite[Thms. XVIII.3.1 and XVIII.3.2]{DautrayLions} or \cite[Thm. 11.3]{RenardyRogers}.
\end{proof}

For $  \amax \geq  t\geq s \geq 0$, we denote by $U(t, s)$ the evolution operator of linear problem \eqref{eq:linChar} without inhomogeneity $f$, which takes an initial condition $v_0$ at initial time $s$ to the weak solution $v(t) = U(t, s) v_0$ of \eqref{eq:linChar} at time $t$. Similarly, by $S(t, s) f$ we denote the solution to \eqref{eq:linChar} with the inhomogeneity $f$ at time $t$ with the initial condition zero at time $s$. Furthermore, we will consider $U'(t, s) v_0$ and $S'(t, s) f$  as the corresponding weak derivative. According to Lemma \ref{lem:well_linear} and using standard arguments for parabolic equations, we can conclude that that 
\begin{align}
	U(\cdot, s) &\in \mathcal{L}(H, W((s, \amax), V, V')),\\
	U'(\cdot, s) &\in \mathcal{L}(H, L^2((s, \amax), V')),\\
	S(\cdot, s) &\in \mathcal{L}(L^2((s, \amax), V'), W((s, \amax), V, V')),\\
	S'(\cdot, s) &\in \mathcal{L}(L^2((s, \amax), V'), L^2((s, \amax), V')).
\end{align} 
From the embedding $W((s, \amax)), V, V') \hookrightarrow C([s, \amax], H)$ we conclude that $U(t, s) \in \mathcal{L}(H, H)$ and $S(t, s) \in \mathcal{L}(L^2((s, \amax), V'), H)$. It can easily be shown that the norms of all of these operators can be estimated above by a constant that does not depend on $t$ and $s$ but only on $\amax$.

\begin{Bem} \label{Duhamel}
    Not that in the case where even $f \in L^2((\max\set{0, -t_0}, \amax), H)$, we can use Duhamel's principle (variation of constants) to show that a solution to \eqref{eq:linChar} with initial condition $v(s) = v_0 \in H$ can also be expressed as \begin{equation}
        v(t) = U(t, s) v_0 + \int_s^t U(t, r) f(r) \dd{r}.
    \end{equation}
    In other words, $S(t, s) f = \int_s^t U(t, r) f(r) \dd{r}$ in this case.
\end{Bem}

Now, for any given $y_0 \in L^2((0, \amax), H)$ and $B \in L^2((0, T), H)$ in \eqref{eq:linear_with_age}, we calculate  back to the $t$ and $a$ variables, and  formally obtain 
 for almost every $(t,a)\in (0,T)\times \mathcal{I}$ that 
\begin{equation}
	y(t, a, x) := \begin{cases} \label{eq:Bfixed}
		U(a,0) B(t-a) + S(a, 0) \restr{f}_{\chak(t-a)} & t > a,\\
		U(a,a-t) y_0(a-t) + S(a, a-t) \restr{f}_{\chak(t-a)} & t \leq a.
	\end{cases} 
\end{equation}
We then will show that $y$ defined by \eqref{eq:Bfixed} is indeed the weak solution of \eqref{eq:linear_with_age}.  To achieve this, we first demonstrate that the function $v$ defined by 
 \begin{equation}
	v(t, a, x) \coloneqq \begin{cases} \label{eq:deltaExpression}
		U'(a,0) B(t-a) + S'(a, 0) \restr{f}_{\chak(t-a)} & t > a\\
		U'(a,a-t) y_0(a-t) + S'(a, a-t) \restr{f}_{\chak(t-a)} & t \leq a,
	\end{cases} 
\end{equation}
for almost every $(t,a)\in (0,T)\times \mathcal{I}$  satisfies  $v = \delta y = (\partial_t + \partial_a) y$ in a weak sense. 

Before proceeding, we must verify that these expressions are well-defined. Specifically, they depend on representatives of $B$, $y_0$, and $f$, which are only defined modulo null sets. The following lemma shows that this dependence does not lead to any issue, as different choices of representatives will lead to expressions that also only differ on a null set.

\begin{Lem} \label{lem:NullSets}
The following statements hold:
\begin{enumerate}
    \item Let $N_1 \subset [0, T] \times \bar{\mathcal{I}}$ be  a null set. Then the   set \begin{equation}
        \set{t_0 \in [-\amax, T] \where N_1 \cap \chak(t_0) \text{ has nonzero measure in} \chak(t_0)}
    \end{equation}
    is a null set in $[-\amax, T]$.
    \item Let $N_2 \subset [-\amax, T]$ be a null set. Then the set $\mathcal{N} \coloneqq \bigcup_{t_0 \in N_2} \chak(t_0)$ is a null set in $[0, T] \times \bar{\mathcal{I}}$.
\end{enumerate}
\end{Lem}
\begin{proof}
    Both claims follow using Fubini's theorem and the fact that  \begin{equation}
        \int_{[0, T] \times \mathcal{I}} \phi(t, a) \dd{(t, a)} = \int_{-\amax}^T \int_{\chak(t_0)} \restr{\phi}_{\chak(t_0)}(h) \dd{h} \dd{t_0}.
    \end{equation}
    Due to \begin{equation}
        0 = \int_{[0, T] \times \mathcal{I}} \eins_{N_1}(t, a) \dd{(t, a)} = \int_{-\amax}^T \int_{\chak(t_0)} \eins_{N_1 \cap \chak(t_0)}(h) \dd{h} \dd{t_0},
    \end{equation}
    we can conclude that $\int_{\chak(t_0)} \eins_{N_1 \cap \chak(t_0)}(h) \dd{h} = 0$ for almost all $t_0 \in (-\amax,T)$, which shows the first claim. Further, the second claim follows from the fact that  
    \begin{equation}
        \int_{[0, T] \times \mathcal{I}} \eins_\mathcal{N}(t, a) \dd{(t, a)} = \int_{N_2} \int_{\chak(t_0)} 1 \dd{h} \dd{t_0} = 0.
    \end{equation}
\end{proof}
\begin{Lem}
\label{Lem:well_lin}
Suppose that Assumption \ref{assum:Lin} holds. Then the functions $y$ from \eqref{eq:Bfixed} and $v$ from \eqref{eq:deltaExpression} satisfy $y \in L^2((0, T) \times \mathcal{I}, V) \cap L^\infty((0, T), L^2(\mathcal{I}, H))$ and $v \in L^2((0, T) \times \mathcal{I}, V')$. Further, we have the estimate
 \begin{equation}
 \label{eq:yEstimate}
 \begin{split}
    \norm{y}_{L^2((0, T) \times \mathcal{I}, V)}^2 &+ \norm{y}_{L^\infty((0, T), L^2(\mathcal{I}, H))}^2 + \norm{v}_{L^2((0, T) \times \mathcal{I}, V')}^2 \\ &\leqc \left( \norm{y_0}_{L^2(\mathcal{I}, H)}^2 + \norm{B}_{L^2((0, T), H)}^2 + \norm{f}_{L^2((0, T) \times \mathcal{I}, V')}^2 \right) ,
    \end{split}
\end{equation}
with a constant independent of $y_0$, $f$ and $B$.
\end{Lem}
\begin{proof}
	 The wellposedness of $y$ and $v$ follows by Lemma \ref{lem:NullSets}, and the regularity will follow from \eqref{eq:yEstimate}, which is what we are going to show next. By transforming along characteristic lines and applying Fubini's theorem we obtain that 
   {\small \begin{align}
		&\quad \norm{y}_{L^2((0, T) \times \mathcal{I}, V)}^2 + \norm{v}_{L^2((0, T) \times \mathcal{I}, V')}^2 = \int_0^T \int_0^\amax \norm{y(t, a)}_V^2 + \norm{v(t, a)}_{V'}^2 \dd{a} \dd{t}\\
		&= \int_{-\amax}^T \int_{t_0^-}^{\min\set{\amax, T - t_0}} \norm{y(t_0 + h, h)}_V^2 + \norm{v(t_0 + h, h)}_{V'}^2 \dd{h} \dd{t_0}\\
		&= \int_{\mathllap{-\amax}}^0 \int_{t_0^-}^{\mathrlap{\min\set{\amax, T - t_0}}} \norm{U(h, -t_0) y_0(-t_0) + S(h, -t_0) \restr{f}_{\chak(t_0)}}_V^2 + \norm{U'(h, -t_0) y_0(-t_0) + S'(h, -t_0) \restr{f}_{\chak(t_0)}}_{V'}^2 \dd{h} \dd{t_0}\\*
		&\quad + \int_0^T \int_{t_0^-}^{\mathrlap{\min\set{\amax, T - t_0}}} \norm{U(h, 0) B(t_0) + S(h, 0) \restr{f}_{\chak(t_0)}}_V^2 + \norm{U'(h, 0) B(t_0) + S'(h, 0) \restr{f}_{\chak(t_0)}}_{V'}^2 \dd{h} \dd{t_0}\\
		&\leqc \int_{-\amax}^0 \norm{y_0(-t_0)}_H^2 + \norm{\restr{f}_{\chak(t_0)}}_{L^2(\chak(t_0), V')}^2 \dd{t_0} + \int_0^T \norm{B(t_0)}_H^2 + \norm{\restr{f}_{\chak(t_0)}}_{L^2(\chak(t_0), V')}^2\dd{t_0}\\
		& = \norm{y_0}_{L^2(\mathcal{I}, H)}^2 + \norm{B}_{L^2([0, T], H)}^2 + \norm{f}_{L^2([0, T] \times \mathcal{I}, V')}^2.
	\end{align}}
    Similarly, we can write for almost every $t \in (0,T)$ that 
    \begin{align}
	   &\quad \norm{y(t)}_{L^2(\mathcal{I}, H)}^2 = \int_0^\amax \norm{y(t, a)}_H^2 \dd{a}\\
	   &= \int_0^{\mathclap{\min(t, \amax)}} \norm{U(a,0) B(t-a) + S(a, 0) \restr{f}_{\chak(t-a)}}_H^2 \dd{a}\\
      &\quad + \int_{\mathclap{\min(t, \amax)}}^\amax \norm{U(a,a-t) y_0(a-t) + S(a, a-t) \restr{f}_{\chak(t-a)}}_H^2 \dd{a}\\
	   &\leqc \int_0^t \norm{B(t-a)}^2_H \dd{a} + \int_0^t \int_0^\amax \norm{f(s, a)}^2_{V'} \dd{a} \dd{s} + \int_0^\amax \norm{y_0(a)}^2_H \dd{a}.
    \end{align}
    Thus, we have completed the verification of \eqref{eq:yEstimate}. 
\end{proof}

\begin{Thm} \label{Thm:yCont}
    The function $t \mapsto y(t, \cdot, \cdot)$ is a continuous mapping from $[0, T]$ into $\hilbert = L^2(\mathcal{I}, H)$, i.e. $y \in C([0, T], L^2(\mathcal{I}, H))$ and we have \begin{equation}
        \norm{y}_{C([0, T], \hilbert)}^2 \leqc \norm{y_0}_{L^2(\mathcal{I}, H)}^2 + \norm{B}_{L^2((0, T), H)}^2 + \norm{f}_{L^2((0, T) \times \mathcal{I}, V')}^2.
    \end{equation}
\end{Thm}

\begin{figure}[h]
    \centering
    \begin{subfigure}{.4\textwidth}
        \begin{tikzpicture}[scale = .5]
        \draw [->] (0,-7.5) -- (0,4.4) node[above]{$a$};
        \draw [->] (0,0) node[left]{0} -- (7.5,0) node[right]{$t$};
        \draw[dotted] (7,-3) -- (7,0) -- (7,4) node[above]{$T$} -- (0,4);
        \draw (0,4) node[left]{$\amax$} -- (7,-3) -- (7,-7) -- (0, 0);
        \draw [blue, thick] (0,3) -- (1,3); \draw[dotted] (1,3) -- (7,3);
        \draw [blue, thick] (0,2) -- (2,2); \draw[dotted] (2,2) -- (7,2);
        \draw [blue, thick] (0,1) -- (3,1); \draw[dotted] (3,1) -- (7,1);
        \draw [blue, thick] (0,0) -- (4,0); \draw[dotted] (4,0) -- (7,0);
        \draw [dotted] (0,-1) -- (1,-1); \draw [blue, thick] (1,-1) -- (5,-1); \draw[dotted] (5,-1) -- (7,-1);
        \draw [dotted] (0,-2) -- (2,-2); \draw [blue, thick] (2,-2) -- (6,-2); \draw[dotted] (6,-2) -- (7,-2);
        \draw [dotted] (0,-3) -- (3,-3); \draw [blue, thick] (3,-3) -- (7,-3);
        \draw [dotted] (0,-4) -- (4,-4); \draw [blue, thick] (4,-4) -- (7,-4);
        \draw [dotted] (0,-5) -- (5,-5); \draw [blue, thick] (5,-5) -- (7,-5);
        \draw [dotted] (0,-6) -- (6,-6); \draw [blue, thick] (6,-6) -- (7,-6);
        \draw [dotted] (0,-7) node[left]{$-T$} -- (7,-7);
	    \end{tikzpicture}
        \caption{Definition of $\tilde{y}$}
        \label{subfig:yTilde}
    \end{subfigure}
    \hspace{2cm}
    \begin{subfigure}{.4\textwidth}
        \begin{tikzpicture}[scale = .5]
        \draw [->] (0,-7.5) -- (0,4.4) node[above]{$a$};
        \draw [->] (0,0) node[left]{0} -- (7.5,0) node[right]{$t$};
        \draw (0,4) node[left]{$\amax$} -- (7,4) -- (7,0) node[below]{$T$};
        \draw [blue, thick] (0,3) -- (1,4);
        \draw [blue, thick] (0,2) -- (2,4);
        \draw [blue, thick] (0,1) -- (3,4);
        \draw [blue, thick] (0,0) -- (4,4);
        \draw [dotted] (0,-1) -- (1,0); \draw [blue, thick] (1,0) -- (5,4);
        \draw [dotted] (0,-2) -- (2,0); \draw [blue, thick] (2,0) -- (6,4);
        \draw [dotted] (0,-3) -- (3,0); \draw [blue, thick] (3,0) -- (7,4);
        \draw [dotted] (0,-4) -- (4,0); \draw [blue, thick] (4,0) -- (7,3);
        \draw [dotted] (0,-5) -- (5,0); \draw [blue, thick] (5,0) -- (7,2);
        \draw [dotted] (0,-6) -- (6,0); \draw [blue, thick] (6,0) -- (7,1);
        \draw [dotted] (0,-7) node[left]{$-T$} -- (7,0);
	    \end{tikzpicture}
        \caption{Definition of $y$}
        \label{subfig:y}
    \end{subfigure}
        
     \caption{Illustration of the relationship between $\tilde{y}$ and $y$, and the operation of the shift and restriction operators. Subfigure \ref{subfig:yTilde} shows how $\tilde{y}$ is defined: the desired evolution operators are applied on the blue lines, while the definition on the dashed lines to ensure a pointwise continuous function. By applying the shift and restricting the domain for $a$, we obtain subfigure \ref{subfig:y}, which represents the desired function $y$ (cf. Fig. \ref{fig:CharakteristikTkA}).}
    
    \label{fig:ShiftAction}
\end{figure}
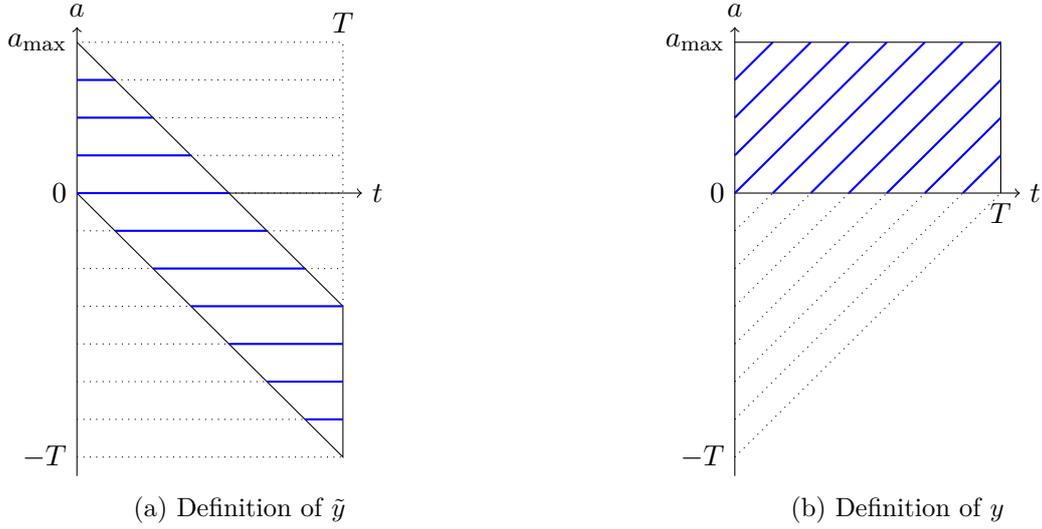

\begin{proof}
    First consider the function $\tilde{y}: [0, T] \times [-T, \amax] \to H$ defined by \begin{equation}
        \tilde{y}(t, a) \coloneqq \begin{cases}
		U(a+t,0) B(-a) + S(a+t, 0) \restr{f}_{\chak(-a)} & a < 0,\\
		U(a+t,a) y_0(a) + S(a+t, a) \restr{f}_{\chak(-a)} & a \geq 0,
	\end{cases} 
    \end{equation}
    where we generalize $U(t, s) \coloneqq U(\min\{t, \amax\}, \max\{s, 0\})$ and similar for $S$. Note that also with the new definition we have that $t \mapsto U(t, s) \phi \in C([s, \infty), H)$ for all $t$, $s \in \R$ and all $\phi \in H$ and we can estimate $\norm{U(t, s)}_{\mathcal{L}(H, H)}$ above uniformly in $t$ and $s$. Similar statements hold for $S$. For a motivation on how to define $\tilde{y}$, we refer to Fig. \ref{fig:ShiftAction}. Since for all $t \in [0, T]$ we have \begin{equation} \label{eq:DomConv}
        \norm{\tilde{y}(t, a)}_H^2 \leqc \sup_{s \geq r} \norm{S(s, r)}^2 \cdot \norm{\restr{f}_{\chak(-a)}}_{L^2(\chak(-a), V')}^2 + \sup_{s \geq r} \norm{U(s, r)}^2 \cdot \begin{cases}
            \norm{B(-a)}_H^2, & a<0\\
            \norm{y_0(a)}_H^2, & a \geq 0
        \end{cases},
    \end{equation}
    there holds $\norm{\tilde{y}(t)}_{L^2((-T, \amax), H)}^2 \leqc \norm{y_0}_{L^2(\mathcal{I}, H)}^2 + \norm{B}_{L^2((0, T), H)}^2 + \norm{f}_{L^2((0, T) \times \mathcal{I}, V')}^2$ and thus $\tilde{y} \in L^\infty((0,T), L^2((-T, \amax), H))$. We claim that even $\tilde{y} \in C([0, T], L^2((-T, \amax), H))$. To show this, let $(t_n)_{n \in \N}$ a sequence in $[0, T]$ that converges to some $\hat{t}$. From the continuity of $U$ and $S$ it follows that $\tilde{y}(t_n) \to \tilde{y}(\hat{t})$ pointwise for all $a \in [-T, \amax]$. Since in estimate \eqref{eq:DomConv} the right-hand side is an element of $L^1((-T, \amax), H)$ when interpreted as a function of $a$, we can invoke the dominated convergence theorem  \cite[Prop. 1.2.5]{AnalysisInBanach} to obtain \begin{equation}
    \lim_{n \to \infty} \int_{-T}^\amax \norm{y(t_n, a) - y(\hat{t}, a)}_H^2 \dd{a} = 0.
    \end{equation}
    This just means that $\tilde{y}(t_n) \to \tilde{y}(\hat{t})$ in $L^2((-T, \amax), H))$, and hence shows the continuity of $\tilde{y}$.

    Next, define the restriction operator \begin{equation}
        R \in \mathcal{L}(L^2((-T, \amax), H), L^2((0, \amax), H)), \quad \phi \mapsto \restr{\phi}_{(0, \amax)}.
    \end{equation} and for any $t \in [0, T]$ the age-shift operator \begin{equation}
        S(t) \in \mathcal{L}(L^2((-T, \amax), H)), \quad (S(t) \phi)(a) = \begin{cases}
            \phi(a-t), & a-t \geq -T,\\
            0, & a-t < -T.
        \end{cases}
    \end{equation}  The continuity of these operators is well known. Our next step is to show that \begin{equation}
        y(t) = R S(t) \tilde{y}(t),
    \end{equation}
    which concludes the proof since we have written $y$ as a composition of continuous functions. In fact, for any $t \in [0, T]$ and $a \in [0, \amax]$ we have $a-t \geq -T$ and hence \begin{equation}
        (S(t) \tilde{y}(t))(a) = \begin{cases}
		U(a-t+t,0) B(t-a) + S(a-t+t, 0) \restr{f}_{\chak(t-a)} & a-t < 0,\\
		U(a-t+t,a) y_0(a-t) + S(a-t+t, a) \restr{f}_{\chak(t-a)} & a-t \geq 0,
	\end{cases} 
    \end{equation}
    which is just the expression for $y$ from eq. \eqref{eq:Bfixed}. The norm estimate has already been established in Lemma \ref{Lem:well_lin}.
\end{proof}

\begin{Bem}\label{rem:Randbed}
    Tracking the convergence for $t_n \to 0$ shows that, in fact, $y(t=0, \cdot, \cdot) = y_0$, the left-hand side being a valid expression as by the above theorem. By swapping the roles of $t$ and $a$ in the previous proof we can also show that $a \mapsto y(\cdot, a, \cdot)$ is a continuous mapping from $\bar{\mathcal{I}}$ into $L^2([0, T], H)$ and it holds that $y(\cdot, a, \cdot) = B$.
\end{Bem}

\begin{Lem} \label{thm:Loesungsbegriff}
	Suppose that Assumption \ref{assum:Lin} holds. Then the function $v$, defined in \eqref{eq:deltaExpression}, represents the weak derivative of $y$ defined in \eqref{eq:Bfixed} in the time-age space. That is, it holds that $v = \delta y = \partial_t y + \partial_a y$, in the sense of $V'$-valued functions on $(0, T) \times \mathcal{I}$.
\end{Lem}
\begin{proof}
	Let $\phi \in C_c^\infty((0, T) \times \mathcal{I}, V')$, then use the same transformation to characteristics as in Lemma \ref{Lem:well_lin} we can write
    \begin{align}
		&\quad \int_0^T \int_0^\amax \skpr{\delta \phi(t, a), y(t, a)}_{V'} + \skpr{\phi(t, a), v(t, a)}_{V'}  \dd{a} \dd{t}\\
		&= \int_{-\amax}^T \int_{t_0^-}^{\min\set{\amax, T - t_0}} \skpr{\delta \phi(t_0 + h, h), y(t_0 + h, h)}_{V'} + \skpr{\phi(t_0 + h, h), v(t_0 + h, h)}_{V'} \dd{h} \dd{t_0}\\
		&= \int_{-\amax}^0 \int_{t_0^-}^{\min\set{\amax, T - t_0}} \skpr{\dv{h} \phi(t_0 + h, h), U(h, -t_0) y_0(-t_0)}_{V'} + \skpr{\phi(t_0 + h, h), U'(h, -t_0) y_0(-t_0)}_{V'}\\
		&\quad + \skpr{\dv{h} \phi(t_0 + h, h), S(h, -t_0) \restr{f}_{\chak(t_0)}}_{V'} + \skpr{\phi(t_0 + h, h), S'(h, -t_0) \restr{f}_{\chak(t_0)}}_{V'} \dd{h} \dd{t_0}\\
		&\quad + \int_0^T \int_{t_0^-}^{\min\set{\amax, T - t_0}} \skpr{\dv{h} \phi(t_0 + h, h), U(h, 0) B(t_0)}_{V'} + \skpr{\phi(t_0 + h, h), U'(h, 0) B(t_0)}_{V'}\\
		&\quad + \skpr{\dv{h} \phi(t_0 + h, h), S(h, 0) \restr{f}_{\chak(t_0)}}_{V'} + \skpr{\phi(t_0 + h, h), S'(h, 0) \restr{f}_{\chak(t_0)}}_{V'} \dd{h} \dd{t_0}\\
		&= 0,
	\end{align}
	where in the third line we have used the representations of $y$ and $v$ given in \eqref{eq:Bfixed} and \eqref{eq:deltaExpression}, and the last equality holds due to the definition of weak derivatives in intervals. Note that $\restr{\phi}_{\chak(t_0)}$ is a test function again for all values of $t_0$.
\end{proof}
\begin{Kor}[Weak solution] \label{cor:WeakSolution}
From Lemmas \ref{Lem:well_lin} and \ref{thm:Loesungsbegriff} we conclude that $y$ is a solution of \eqref{eq:linear_with_age} in the sense that it holds for almost all $(t, a) \in (0,T)\times \mathcal{I}$ that 
\begin{equation}
	\skp{\delta y(t, a), v}_{V',V} + \skpr{L(a) y(t, a), v}_H + \skpr{\sigma(a) \nabla y(t, a), \nabla v}_{H^d} = \skp{f(t, a), v}_{V',V} \quad  \text{ for all }  v \in V 
\end{equation}
as well as the initial conditions $y(t=0, \cdot, \cdot) = y_0$ and $y(\cdot, a = 0, \cdot) = B$ from Remark \ref{rem:Randbed}.
\end{Kor}

\subsection{The Linearized Equation with Implicit Birth Law}

In this section, we incorporate the birth equation
\begin{equation}
    B(t, x) = \int_0^\amax \beta(\alpha, x) y(t, \alpha, x) \dd{\alpha}
    \end{equation}
into \eqref{eq:linear_with_age} in place of $B$, which yields
\begin{equation}\label{eq:linear_implicit}
		\begin{split}
		    \delta y + L(a, x) y  - \sigma(a) \laplace y &= f,\\*
		y(t = 0) = y_0, \quad y(a = 0) = \int_0^\amax \beta(\alpha, x) y(t, \alpha, x) \dd{\alpha} &\eqqcolon B\\*
		\partial_\nu y(x \in \partial \Omega) &= 0.
		\end{split}
\end{equation} 
\begin{Assump} \label{assum:LinImp}
For the remainder of this section, we adopt the same assumptions as stated in Assumption \ref{assum:Lin}, except for the first line, which we modify as follows:
\begin{enumerate}
\item  $\beta \in C(\bar{\mathcal{I}}, L^\infty(\Omega)^{n \times n})$, $f \in L^2((0, T) \times \mathcal{I}, V')$ and $y_0 \in L^2(\mathcal{I}, H)$.
\end{enumerate}
\end{Assump}

Plugging the implicit birth law into \eqref{eq:Bfixed} yields
\begin{equation}
	B(t, x) = \int_0^\amax \beta(\alpha, x) \begin{cases}
		U(\alpha,0) B(t-a) + S(\alpha, 0) \restr{f}_{\chak(t-\alpha)} & t > \alpha\\
		U(\alpha,\alpha-t) y_0(\alpha-t) + S(\alpha, \alpha-t) \restr{f}_{\chak(t-\alpha)} & t \leq \alpha
	\end{cases} \dd{\alpha}.
\end{equation}
Splitting the integral into $\int_0^{\min(t, \amax)} + \int_{\min(t, \amax)}^\amax$ yields
\begin{equation}\label{eq:voltera}
    \begin{split}
        B(t, x) &= \int_0^{\min(t, \amax)} \beta(\alpha, x) U(\alpha, 0) B(t - \alpha, x) \dd{\alpha} + \int_0^{\min(t, \amax)} \beta(\alpha, x) S(\alpha, 0) \restr{f}_{\chak(t-\alpha)} \dd{\alpha}\\
	&\quad + \int_{\min(t, \amax)}^\amax \beta(\alpha, x) U(\alpha, \alpha - t) y_0(\alpha-t) \dd{\alpha} + \int_{\min(t, \amax)}^\amax \beta(\alpha, x) S(a, a-t) \restr{f}_{\chak(t-a)} \dd{\alpha}
    \end{split}
\end{equation}
which is a fixed point Volterra equation for $B$.
\begin{Thm} \label{Thm:Bexistence}
Suppose that Assumption \ref{assum:LinImp} holds. Then there exists a unique solution $B = B(y_0, f)\in L^2((0, T), H)$ to the Volterra equation \eqref{eq:voltera}, satisfying the estimate 
\begin{equation}
\norm{B}_{L^2((0, T), H)} \leqc \norm{y_0}_{L^2(\mathcal{I}, H)} + \norm{f}_{L^2((0, T) \times \mathcal{I}, V')},
\end{equation}
with a constant independent of $y_0$ and $f$.
\end{Thm}
\begin{proof}
    It is straightforward to verify that $\alpha \mapsto \beta(\alpha) U(\alpha, 0)$ belongs to  $C(\mathcal{I}, L(H))$ and that \begin{multline}
        t \mapsto \int_0^{\min(t, \amax)} \beta(\alpha, x) S(\alpha, 0) \restr{f}_{\chak(t-\alpha)} \dd{\alpha}\\
        + \int_{\min(t, \amax)}^\amax \beta(\alpha, x) U(\alpha, \alpha - t) y_0(\alpha-t) \dd{\alpha} + \int_{\min(t, \amax)}^\amax \beta(\alpha, x) S(\alpha, \alpha-t) \restr{f}_{\chak(t-\alpha)} \dd{\alpha}
    \end{multline}
    is an element of $L^2((0, T), H)$. The claim then follows by \cite[Cor. 0.2]{Pruess}.
\end{proof}

\begin{Bem} \label{rem:Semigroup}
Suppose that $f =0$, and for given $y_0 \in L^2(\mathcal{I},H)$, let $B_{y_0}$ be the solution of \eqref{eq:voltera}. Then, if we define 
     \begin{equation} 
	(T(t) y_0)(a, x) := \begin{cases} \label{eq:Bimplicit}
		U(a, 0) B_{y_0}(t-a), & t > a\\
		U(a, a-t) y_0(a-t), & t \leq a
	\end{cases},
\end{equation} 
 one can show that $(T(t))_{t \geq 0}$ is a $C^0$ semigroup of operators on $L^2([0, \amax], H) = \hilbert$ corresponding to the Volterra equation \eqref{eq:voltera}, see e.g. \cite[Thm. 4]{Webb}. As a consequence, there exist constants $M \geq 1$, $\omega \in \R$ such that $\norm{T(t)} \leq Me^{\omega t}$.
\end{Bem}

We are now ready to present the main result of this subsection: the well-posedness of the linearized equation \eqref{eq:linear_implicit}. By combining the estimates from \eqref{eq:yEstimate} and  Theorem \ref{Thm:yCont} with the results for $B$ established in Theorem \ref{Thm:Bexistence}, we arrive at the following theorem:

\begin{Thm}\label{thm:linReg} Suppose that Assumption \ref{assum:LinImp} holds. Then there exists a unique weak solution $y \in L^2((0, T) \times \mathcal{I}, V) \cap C([0, T], \hilbert)$ with $\delta y \in L^2((0, T) \times \mathcal{I}, V')$ to \eqref{eq:linear_implicit} in the sense given Corollary \ref{cor:WeakSolution} satisfying the estimate \begin{equation}\label{eq:AbschaetzungCharDarstellung}
		\norm{y}_{L^2((0, T) \times \mathcal{I}, V)}^2 + \norm{y}_{C([0, T],L^2(\mathcal{I}, H))}^2 + \norm{\delta y}_{L^2((0, T) \times \mathcal{I}, V')}^2 \leqc \norm{y_0}_{L^2(\mathcal{I}, H)}^2 + \norm{f}_{L^2((0, T) \times \mathcal{I}, V')}^2.
	\end{equation}
\end{Thm}

\subsection{The Nonlinear Equation}
In this section, we investigate the well-posedness of the original nonlinear model \eqref{eq:StateEquation}. To achieve this, we consider the following assumptions.

\begin{Assump}
\label{assump:non}
    We assume the following  
    \begin{enumerate}
    \item Assumption \ref{assum:LinImp} holds
	\item For every $u$ satisfying \eqref{eq:Control_C},  the control input $K(u) \in L^\infty((0, T) \times \mathcal{I} \times \Omega)$ depends linearly on $u \in L^\infty((0, T) \times \mathcal{I} \times \Omega)$ and it does not depend explicitly on time, age and space. 
	\item $\Lambda$ is a well-defined operator. That is,  for all $k \in L^\infty(\mathcal{I} \times \Omega,  \hilbert)$, the mapping  \begin{equation}
		(a, x) \in \mathcal{I} \times \Omega \mapsto \skpr{k(a, \cdot, x, \cdot), y}_\hilbert   \quad \text{ for all } y \in H  
	\end{equation}
	is well-defined and belongs $L^\infty(\mathcal{I} \times \Omega)$, meaning that it can be again multiplied with elements from $\hilbert$ and the result being an element of $\hilbert$ as well. Furthermore, this structure allows an estimate of the form $\norm{\Lambda(y_1) y_2}_\hilbert \leq c(k) \norm{y_1}_\hilbert \norm{y_2}_\hilbert$ for some constant $c$ depending on $k$.
\end{enumerate}\end{Assump}

Integrating the estimate on $\Lambda$ with $y_1$, $y_2 \in L^2((0, T), \hilbert)$ over the interval $[0, t]$ where $t \in [0, T]$ yields \begin{align}\label{eq:Lambdaestimate}
	\norm{\Lambda(y_1) y_2}_{L^2((0, t), \hilbert)} &\leqc \norm{y_1}_{L^\infty((0, t), \hilbert)} \norm{y_2}_{L^2((0, t), \hilbert)},\\
    \intertext{or, respectively}
    \norm{\Lambda(y_1) y_2}_{L^2((0, t), \hilbert)} &\leqc \norm{y_1}_{L^2((0, t), \hilbert)} \norm{y_2}_{L^\infty((0, t), \hilbert)}.
\end{align}
This shows that for $y \in C([0, T], \hilbert)$, there is $\Lambda(y) y \in L^2((0, T), \hilbert)$. According to Thm. \ref{thm:linReg} this is regular enough for $f$ in eq. \eqref{eq:linear_implicit} to be replaced by $\Lambda(y) y$. Formally this is done with a fixed-point argument.

\begin{Thm}
\label{Thm:existence_Non}
Suppose that Assumption \ref{assump:non} holds. Then there exists a $T^* \leq T$ that only depends on $\bar{u}$ from  \eqref{eq:Control_C} such that for almost every  $t \in (0, T^*)$ the state equation 
\begin{align}
		\delta y + L(a, x) y + \Lambda(a, x, y) y + K(u) y &= \sigma(a) \laplace y,\\
		y(t = 0) &= y_0,\\
        y(a = 0) &= \int_0^\amax \beta(\alpha, x) y(t, \alpha, x) \dd{\alpha},\\
		\partial_\nu y(x \in \partial \Omega) &= 0
	\end{align}
	has a unique weak solution $y \in L^2((0, T^*) \times \mathcal{I}, V) \cap C([0, T^*], \hilbert)$ with $\delta y \in L^2((0, T^*) \times \mathcal{I}, V')$ satisfying the weak formulation 
    \begin{equation}
    \label{eq:weak_formualation}
	\skp{\delta y, v}_{V',V} + \skpr{L y, v}_H + \skpr{\Lambda(y) y, v}_H + \skpr{\sigma(a) \nabla y, \nabla v}_{H^d} = 0 \quad  \text{ for all }  v \in V 
\end{equation}
    for almost all $t$ and $a$. Furthermore, we have an energy estimate of the form
    \begin{equation}
    \label{eq:Ener_Estimate_for_nonlinear_state}
	    \norm{y}_{L^2((0, T^*) \times \mathcal{I}, V)}^2 + \norm{y}_{C([0, T^*], \hilbert)}^2 + \norm{\delta y}_{L^2((0, T^*) \times \mathcal{I}, V')}^2 \leqc \norm{y_0}_\hilbert^2.
	\end{equation}
\end{Thm}
\begin{proof} 

The proof is based on Banach's fixed-point argument and is inspired by \cite[Thm. 14.2 and Lem. 14.3]{Smoller}. We start by showing uniqueness. If $y^1$ is a solution of the nonlinear equation to the initial value $y_0^1$ and the control $u_1$ and $y^2$ a solution to the initial value $y_0^2$ and control $u_2$, then the difference $w \coloneqq y^1 - y^2$ satisfies \begin{align}
		\delta w + L(a, x) w  &= \sigma(a) \laplace w + \Lambda(y^2) y^2 - \Lambda(y^1) y^1 + K(u_2) y^2 - K(u_1) y^1,\\
		w(t = 0) &= y_0^1 - y_0^2,\\
        w(a = 0) &= \int_0^\amax \beta(\alpha, x) w(t, \alpha, x) \dd{\alpha},\\
		\partial_\nu w(x \in \partial \Omega) &= 0
	\end{align}
    Thus, by Theorem \ref{thm:linReg} and using \eqref{eq:Lambdaestimate}, we can write the estimate 
    \begin{align} \label{eq:uniquenessEstimate}
		\MoveEqLeft\norm{w}_{L^2((0, T) \times \mathcal{I}, V)}^2 + \norm{w(t)}_\hilbert^2 + \norm{\delta w}_{L^2((0, T) \times \mathcal{I}, V')}^2\\
		&\leqc \norm{y^1_0 - y^2_0}_{L^2(\mathcal{I}, H)}^2 + \norm{\Lambda(y^2) y^2 - \Lambda(y^1) y^1}_{L^2((0, T) \times \mathcal{I}, V')}^2 + \norm{K(u_2) y^2 - K(u_1) y^1}_{L^2((0, T) \times \mathcal{I}, V')}^2\\
		&\leqc \norm{y^1_0 - y^2_0}_{L^2(\mathcal{I}, H)}^2 + \norm{\Lambda(y^2) (y^2 - y^1)}_{L^2((0, T) \times \mathcal{I}, V')}^2 + \norm{\Lambda(y^2 - y^1) y^1}_{L^2((0, T) \times \mathcal{I}, V')}^2\\
		&\quad + \norm{K(u_2) (y^2 - y^1)}_{L^2((0, T) \times \mathcal{I}, V')}^2 + \norm{K(u_2 - u_1) y^1}_{L^2((0, T) \times \mathcal{I}, V')}^2\\
		&\leqc \norm{y^1_0 - y^2_0}_{L^2(\mathcal{I}, H)}^2 + \max\set{\norm{y^1}_{L^\infty((0, T), \hilbert)}^2, \norm{y^2}_{L^\infty((0, T), \hilbert)}^2} \norm{w}_{L^2((0, T), \hilbert)}^2\\
		&\quad + \norm{u_2}_{L^\infty} \norm{w}_{L^2((0, T), \hilbert)}^2 + \norm{u_2 - u_1}_{L^\infty}  \norm{y^1}_{L^2((0, T), \hilbert)}^2.
	\end{align}
	Hence, applying Gronwall's inequality establishes the local Lipschitz continuity of the solution operator with respect to the control and the initial function, and for $y_0^1 = y_0^2$ and $u_1 = u_2$ yields uniqueness.   

Next, we establish the existence of a solution. For this purpose, we define the set
\begin{equation}
		\Gamma := \set{y \in C([0, T], \hilbert) \where \norm{y(t) - T(t) y_0}_\hilbert \leq \norm{y_0}_\hilbert \quad \text{for all } t\in [0,T]},
	\end{equation}
	where $T$ is defined as in Rem. \ref{rem:Semigroup}. Note that $\Gamma$ is a closed subset of $L^\infty((0, T), \hilbert)$). For any $y \in \Gamma$ we can estimate 
    \begin{equation} \label{eq:GammaBound}
	    \norm{y(t)}_{\hilbert} \leq \norm{y(t) - T(t) y_0}_{\hilbert} + \norm{T(t) y_0}_{\hilbert} \leq (1 + M e^{\omega t}) \norm{y_0}_\hilbert,
	\end{equation} and, thus, $\Gamma$ is also a bounded subset of $C([0, T], \hilbert)$.
   
    We also define the mapping $\Phi: \Gamma \to \Gamma$, which maps any  $y \in \Gamma$ to the solution $ v:= \Phi(y) $ of the following equation
    \begin{align}
		\delta v + L(a, x) v + \Lambda(a, x, y) y + K(u) y &= \sigma(a) \laplace v,\\
		v(t = 0) = y_0, \quad v(a = 0) = \int_0^\amax \beta(\alpha, x) v(t, \alpha, x) \dd{\alpha} &\\
		\partial_\nu v(x \in \partial \Omega) &= 0.
	\end{align}
	We first show that $\Phi$ is well-defined, that is, it maps $\Gamma$ into itself. Note that for $y \in \Gamma$, the function $w \coloneqq \Phi(y) - T(\cdot) y_0$ satisfies \begin{align}
		\delta w + L(a, x) w + \Lambda(a, x, y) y + K(u) y &= \sigma(a) \laplace w,\\
		w(t = 0) = 0, \quad w(a = 0) = \int_0^\amax \beta(\alpha, x) w(t, \alpha, x) \dd{\alpha} &\\
		\partial_\nu w(x \in \partial \Omega) &= 0.
	\end{align}
	Applying Theorem \ref{thm:linReg} and using \eqref{eq:Lambdaestimate}, we obtain the estimate
    \begin{equation}
        \begin{split}
\MoveEqLeft\norm{w}_{L^2([0, T] \times \mathcal{I}, V)}^2 + \norm{w(t)}_{L^2(\mathcal{I}, H)}^2 + \norm{\delta w}_{L^2([0, T] \times \mathcal{I}, V')}^2 \\
& \leqc \norm{\Lambda(y) y}_{L^2((0, T) \times \mathcal{I}, V')}^2 + \norm{K(u) y}_{L^2((0, T) \times \mathcal{I}, V')}^2 \\
& \leqc \norm{y}_{L^\infty((0, T), \hilbert)}^2 \norm{y}_{L^2((0, T), \hilbert)}^2 + \norm{u}_{L^\infty}^2 \norm{y}_{L^2((0, T), \hilbert)}^2 \\
& \leqc T \norm{y}_{L^\infty((0, T), \hilbert)}^4 + T \norm{u}_{L^\infty}^2 \norm{y}_{L^\infty((0, T), \hilbert)}^2.
\end{split}
    \end{equation}
    This together with \eqref{eq:GammaBound}, \eqref{eq:Control_C}, the fact that $y \in \Gamma$, and an appropriate choice of $T^*$ small enough with $T^*\leq T$, gives us that 
 \begin{equation} \label{eq:PhiInGamma}
	\norm{w}_{L^{\infty}((0,T^*),\hilbert))}^2 \leq \norm{w}_{L^2([0, T^*] \times \mathcal{I}, V)}^2 + \norm{w}_{L^{\infty}((0,T^*), L^2(\mathcal{I}, H))}^2 + \norm{\delta w}_{L^2([0, T^*] \times \mathcal{I}, V')}^2 \leq \norm{y_0}^2.
	\end{equation} 
    Hence, we can conclude that $v = \Phi(y) \in \Gamma$. 
    
    Now, for any $y^1$, $y^2 \in \Gamma$,  we have for $w \coloneqq \Phi(y^1) - \Phi(y^2)$ with the similar computations as in eq. \eqref{eq:uniquenessEstimate} that 
    \begin{align}
		&\quad \norm{w}_{L^{\infty}((0,T),\hilbert)}^2\\
		&\leqc \max\set{\norm{y^1}_{L^\infty((0, T), \hilbert)}^2, \norm{y^2}_{L^\infty((0, T), \hilbert)}^2} \norm{y^2 - y^1}_{L^2((0, T), \hilbert)}^2 + \norm{u}_{L^\infty} \norm{y^2 - y^1}_{L^2((0, T), \hilbert)}^2\\
		&\leqc \qty(\qty(1 + M e^{\omega T})^2 + \bar{u}^2) T \norm{y^1 - y^2}_{L^\infty((0, T), \hilbert)}^2.
	\end{align}
    In the last line, we have used $y^1$, $y^2 \in \Gamma$. Thus, we conclude that for sufficiently small $T^* \leq T$, the mapping $\Phi$ is a contraction. Consequently, the existence of a fixed point $\bar{y} \in \Gamma$ follows from Banach's fixed-point theorem.  
    
    Furthermore, using $\Phi(y) = y$,  we can deduce even higher regularity and the corresponding energy estimates for the solution. More precisely, by utilizing \eqref{eq:PhiInGamma} and setting  $w \coloneqq y - T(\cdot) y_0$ we can write 
    \begin{align}
        &\quad \norm{y}_{L^2([0, T] \times \mathcal{I}, V)}^2 + \norm{y(t)}_{L^2(\mathcal{I}, H)}^2 + \norm{\delta y}_{L^2([0, T] \times \mathcal{I}, V')}^2\\
	    &\leq \norm{w}_{L^2([0, T] \times \mathcal{I}, V)}^2 + \norm{w(t)}_{L^2(\mathcal{I}, H)}^2 + \norm{\delta w}_{L^2([0, T] \times \mathcal{I}, V')}^2\\
        &\quad + \norm{T(\cdot) y_0}_{L^2([0, T] \times \mathcal{I}, V)}^2 + \norm{T(t) y_0}_{L^2(\mathcal{I}, H)}^2 + \norm{\delta T(\cdot) y_0}_{L^2([0, T] \times \mathcal{I}, V')}^2 \leqc \norm{y_0}^2.
	\end{align}
Thus, we are also finished with the derivation of \eqref{eq:Ener_Estimate_for_nonlinear_state}.    
\end{proof}

\section{Existence of an Optimal Control}
This section introduces and investigates the well-posedness of optimal control problems governed by \eqref{eq:StateEquation}. For convenience in numerical experiments and to make the control more realistic, we consider controls of the form  
\begin{equation}\label{eq:FormControl}
	u(t, a, x) = \sum_{i = 1}^M \sum_{j = 1}^N u_{ij}(t) \eins_{\Omega_i}(x) \eins_{[a_{j-1}, a_j]}(a) \quad \text{with }  u(t) := \{u_{ij}(t) \}_{i,j} \in \R^{M\times N}
\end{equation}
Here, the indicator functions $\eins_{\Omega_i}(x)$, with mutually disjoint supports $\Omega_i \subset \Omega$ ($i = 1, \dots, M$), represent vaccination centers, while the indicator functions $\eins_{[a_{j-1}, a_j]}(a)$, with $0 \leq a_0 < a_1 < \dots < a_N \leq \amax$, distinguish different age classes in the vaccination process.
 In this setting, the performance index function \eqref{eq:Ob_function} will take the form  
\begin{equation}
\label{eq:cost}
	J(u, y) = \frac 12 \int_0^T \int_0^\amax \int_\Omega \abs{g \cdot y(t, a, x)}^2 \dd{x} \dd{a} \dd{t} + \frac \alpha2 \int_0^T \|u(t)\|_{\R^{M \times N}}^2 \dd{t}
\end{equation}
where $g$ is a vector of weights. Further, the control constraints \eqref{eq:Control_C} can be rewritten as  
\begin{equation}
\label{eq:C_cosntraints2}
0 \leq  u_{ij}(t) \leq \bar{u} \qquad \text{f.a.a } t\in (0,T) \text{ and all } i=1,\dots,N \text{ and } j =1 ,\dots,N.
\end{equation}
The optimal control problem is then defined as
\begin{equation}
\label{Opt}
\tag{OC}
\inf\set{ J(u,y) \where (u,y)  \text{ satisfy }  \eqref{eq:StateEquation}  \text{ and }  \eqref{eq:C_cosntraints2} }.
\end{equation}
To establish the well-posedness of optimal control, a key challenge in working with nonlinear age- and space-structured models is the lack of compactness. Common compactness results, such as the Rellich–Kondrachov or Aubin–Lions theorems, are not applicable in this setting, as they require information on $y$, $\nabla_x y$, $y_t$, and $y_a$, ensuring that $y$ belongs to an appropriate Sobolev space. However, we only have information on $y$, $\nabla_x y$ and $\delta y$.

To address this issue, we consider more regular controls, specifically $u \in W^{1, p}((0, T),\R^{M\times N})$ with $ 1 < p \leq \infty$. This regularity is achieved by imposing additional control constraints or adding an extra control cost to the performance function.

\begin{Thm}
\label{Thm:existence_OP}
Suppose that one of the following conditions holds:
\begin{enumerate}[\bfseries C1:]
    \item The performance function in \eqref{Opt} is replaced by $J(u,y)+\tfrac{\alpha_d}{2} \norm{\partial_t u}_{L^2((0,T),\R^{M\times N})}^2$ with some $\alpha_d>0$.   
    \item The following additional control constraints are imposed on the problem \eqref{Opt}:
    \begin{equation}
\label{eq:C_cosntraints3}
\| \partial_t u\|_{L^p((0,T),\R^{M\times N})}  \leq \bar{u}_d \qquad \text{f.a.a } t\in (0,T) \text{ and some } p \text{ satisfying }  1 < p \leq \infty. 
\end{equation}
\end{enumerate}
Then, optimal control problem \eqref{Opt} admits a solution.
\end{Thm}
\begin{proof}
The proof is based on the direct method in the calculus of variations. Since $J$ is bounded from below, its infimum exists and is nonnegative. This allows us to choose a minimizing sequence $\{u_n\}_{n}$ with corresponding states $\{y_n\}_{n}$ such that $J(u_n, y_n) \to \inf_{u}J(u,y(u))$, or in the case that C1 holds, that $J(u_n, y_n) + \frac{\alpha_d}{2} \|\partial_t u_n\|_{L^2((0,T),\mathbb{R}^{M\times N})}^2$ converges to its infimum as $n\to \infty$. We consider both cases C1 and C2. In the case where C1 holds, since the cost function is radially unbounded due to the control cost $\|\cdot\|^2_{H^1((0,T),\mathbb{R}^{M\times N})}$, we can infer that the sequence $\{u_n\}_n$ is bounded in $H^1((0, T), \mathbb{R}^{M \times N})$. If C2 holds, the boundedness of $\{u_n\}_n$ in space $W^{1, p}((0, T), \mathbb{R}^{M \times N})$ follows directly. Thus, in either case, there exists a weakly convergent subsequence $u_n \rightharpoonup u^*$ with $u^* \in W^{1, p}((0, T), \mathbb{R}^{M \times N})$, where for C1 we have $p = 2$. We use the same notation for the sequence and its subsequence for convenience.

Since the space $W^{1, p}((0, T), \mathbb{R}^{M \times N})$ with $1 <  p < \infty$  is compactly embedded into $C([0, T], \mathbb{R}^{M \times N})$ (see, e.g., \cite[Thm. 6.3]{AdamsFournier}), we can conclude that $u_n \to u^*$ strongly in $C([0, T], \mathbb{R}^{M \times N})$. Now, it remains to show that the subsequence $\{y_n\}_n$ associated with $\{u_n\}_n$ also converges to $y^*$, the state associated with $u^*$. In a similar manner to \eqref{eq:uniquenessEstimate}, we can write for almost every $t \in (0,T)$ that 
\begin{align}
		&\quad \norm{y_n-y^*}_{L^2((0, T) \times \mathcal{I}, V)}^2 + \norm{y_n(t)-y^*(t)}_{\hilbert}^2 + \norm{\delta y_n - \delta y^*}_{L^2((0, T) \times \mathcal{I}, V')}^2\\
		&\leqc  \max\set{\norm{y_n}_{L^\infty((0, T), \hilbert)}^2, \norm{y^*}_{L^\infty((0, T), \hilbert)}^2} \norm{y_n-y^*}_{L^2((0, T), \hilbert)}^2\\
		&\quad + \norm{u_n}_{L^\infty((0, T), \R^{M \times N})} \norm{y_n-y^*}_{L^2((0, T), \hilbert)}^2 + \norm{u_n - u^*}_{L^\infty((0, T), \R^{M \times N)}}  \norm{y^*}_{L^2((0, T), \hilbert)}^2.
\end{align}
Together with the uniform boundedness of $\{\|u_n\|_{L^\infty((0, T), \mathbb{R}^{M \times N})}\}_n$ and $\{\|y_n\|_{L^\infty((0, T), \mathbb{R}^{M \times N})}\}_n$ (which follows from \eqref{eq:Ener_Estimate_for_nonlinear_state}), and applying Gronwall's inequality, we obtain
\begin{align}
		\norm{y_n-y^*}_{L^{\infty}((0,\infty),\hilbert)}^2 &\leqc \norm{u_n - u^*}_{L^\infty((0, T), \R^{M \times N)}}  \norm{y^*}_{L^2((0, T), \hilbert)}^2,
\end{align}
This shows that $y_n \to y^*$ strongly in $C([0, T], \mathcal{H})$. This strong convergence allows us to conclude that $J(u_n, y_n) \to J(u^*, y^*) = \inf_u J(u,y(u))$, proving the existence of an optimal control in case C2. The existence for case C1 follows from the weak convergence $u_n \rightharpoonup u^*$ in $H^1((0, T), \mathbb{R}^{M \times N})$ and the weak lower semicontinuity of the control cost $\|\partial_t \cdot\|_{L^2((0, T), \mathbb{R}^{M \times N})}$. Therefore, we conclude that
\begin{align}
    J(u^*, y^*) + \frac{\alpha_d}{2} \norm{\partial_t u^*}_{L^2((0,T),\R^{M\times N})}^2 &\leq \liminf_{n \to \infty} \left( J(u_n, y_n) + \frac{\alpha_d}{2} \norm{\partial_t u_n}_{L^2((0,T),\R^{M\times N})}^2 \right) \\
    &  = \inf_u \left(J(u,y(u)) + \frac{\alpha_d}{2} \norm{\partial_t u}_{L^2((0,T),\R^{M\times N})}^2\right).
    \end{align}
Thus, the proof is complete.
\end{proof}

\begin{Bem}
  
It is natural and numerically preferable to consider more general controls, specifically those belonging to $L^2((0,T), \mathbb{R}^{M \times N})$. In this case, the boundedness of the minimizing sequence follows either from the box constraints in \eqref{eq:C_cosntraints2} or the presence of a strictly positive $\alpha$ in \eqref{eq:cost}. As mentioned, the primary challenge lies in the lack of compactness for the state equation, making it unclear whether the weak limit of the minimizing control subsequence corresponds to the weak limit of the associated state subsequence.

In Theorem \ref{Thm:existence_OP}, we addressed this issue by exploiting additional regularity for the control and leveraging a compact embedding, which led to strong convergence. However, for controls in $L^2((0,T), \mathbb{R}^{M \times N})$, compactness becomes an issue for both the control and state, complicating the attainment of strong convergence.

Nevertheless, if the control enters the state equation linearly (rather than bilinearly) and appropriate structural conditions are imposed on the nonlinearity, the existence of optimal control can still be established using only weak convergence. More precisely, we assume that
\begin{itemize}
        \item The state equation has the form 
        \begin{equation}
            \delta y + L(a, x) y + \Lambda(a, x, y) y = \sigma(a) \laplace y + K(u).
        \end{equation}
        That is, the control enters the state equation linearly.
        \item We can extend the nonlocal aspect of $\Lambda$ backwards in time, i.e., we let 
        \begin{equation}
            \begin{split}
	           \Lambda(t, a, x, y)^{hi} &= \int_0^t \int_0^\amax \int_\Omega k^{hij}(t, \theta, a, \alpha, x, \xi) y_j(\theta, \alpha, \xi) \dd{\xi} \dd{\alpha} \dd{\theta}\\& = \skpr{k(t, \cdot, a, \cdot, x, \cdot), I}_{L^2([0, t] \times [0, \amax] \times \Omega)}.
             \end{split}
            \end{equation}
              Moreover, we assume that the kernel $k$ factorizes as follows (for brevity, let $G:= [0, T] \times \mathcal{I} \times \Omega$, $z: = (t, a, x)$, and $\zeta:= (\theta, \alpha, \xi)$)
 \begin{equation}
	           k(z, \zeta) = \sum_{j = 1}^N k_{1j}(z) \odot k_{2j}(\zeta),
            \end{equation} 
            where $\odot$ denotes element-wise multiplication of the 3-tensors $k_{1j}$, $k_{2j}$. Functions of this structure are dense in $L^2(G \times G)$, and if we allow $N = \infty$, they are even dense in $L^\infty(G, L^2(G))$, which is the natural domain for $k$ (See e.g., \cite[Lems. 1.2.19 and 2.1.4]{HNVW}) The temporal nonlocality can be interpreted as infections from germs present in the environment and thus from individuals that have been infectious in the past.       
    \end{itemize}    
     Under these assumptions, the above proof can be carried out without requiring strong convergence.  Indeed, the linear terms are weakly continuous by standard arguments, and the nonlinearity also becomes weakly continuous due to the kernel structure. Specifically, for any sequence $y_n \rightharpoonup y$ in $L^2(G)$ and any $y' \in L^2(G)$, we have   
    \begin{align}
	\skpr{\Lambda(y_n) y_n, y'}_{L^2(G)} &= \skpr{\skpr{k(z, \cdot), y_n}_{L^2(G)} y_n, y'}_{L^2(G)} = \skpr{\skpr{\sum_{j = 1}^N k_{1j}(z) k_{2j}, y_n}_{L^2(G)} y_n, y'}_{L^2(G)}\\
	&= \skpr{\sum_{j = 1}^N k_{1j} \skpr{ k_{2j}, y_n}_{L^2(G)} y_n, y'}_{L^2(G)} = \sum_{j = 1}^N \skpr{k_{2j}, y_n}_{L^2(G)} \skpr{k_{1j} y_n, y'}_{L^2(G)}.
\end{align}
Thus, we can pass to the limit in the weak formulation \eqref{eq:weak_formualation} and complete the proof.
\end{Bem}

\section{First Order Optimality Conditions}
In this section, we derive first-order optimality conditions for the optimal control problem \eqref{Opt}, following \cite[Section 1.7.2]{HPUU}. For this purpose, we define 
\begin{align}
    U &\coloneqq L^2((0,T),\R^{M \times N}),\\
	U^{\text{ad}} &\coloneqq \set{u \in U \where  u \text{ satisfies \eqref{eq:C_cosntraints2}} }   ,\\
	Y &\coloneqq \set{y \in L^2((0, T) \times \mathcal{I}, V) \cap C([0, T], L^2(\mathcal{I}, H)) \cap C(\bar{\mathcal{I}}, L^2((0, T), H)) \where  \delta y \in L^2((0, T) \times \mathcal{I}, V')},\\
	Z &\coloneqq L^2((0, T) \times \mathcal{I}, V') \times L^2(\mathcal{I}, H) \times L^2((0, T), H)
\end{align}
and consider the mappings $J:  Y \times U \to \R$ and $e:  Y \times U \to  Z$ defined by
\begin{align}
	J:\, & (y, u) \in Y \times U \mapsto \frac 12 \iiint \abs{g \cdot y(t, a, x)}^2 \dd{(t, a, x)} + \frac\alpha2 \iiint \abs{u(t, a, x)}^2 \dd{(t, a, x)} \in \R \\
	e:\, &(y, u) \in Y \times U \mapsto \qty(\delta y + L y + \Lambda(y) y + \tilde{K}(u) y - \sigma \laplace y,\, y(t = 0) - y_0,\, y(a = 0) - \int_0^\amax \beta y \dd{a})\in  Z.
\end{align}
Here, for $u \in U$ we let \begin{equation}
    \tilde{K}(u) \coloneqq K\qty(\sum_{i = 1}^M \sum_{j = 1}^N u_{ij}(t) \eins_{\Omega_i}(x) \eins_{[a_{j-1}, a_j]}(a)).
\end{equation}
Then, the optimal control problem \ref{Opt} can be rewritten as 
\begin{equation}
	\inf J(y, u) \qquad \text{ subject to } \qquad e(y, u) = 0.
\end{equation}
We know that $Y$, $Z$ are Banach spaces and $U^{\text{ad}}$ is nonempty, convex and closed in $L^2((0,T),\R^{M \times N})$.  In Theorem \ref{Thm:existence_Non},  we established the existence of a unique solution $y = y(u)$ to the equation $e(y(u), u) = 0$ for all $u \in U^{\text{ad}}$.

 It is straightforward to verify that 
$J$ is continuously Fréchet differentiable. Moreover, apart from the nonlinear terms $\Lambda(y) y$ and $\tilde{K}(u) y$, the mapping  $e$ consists of continuous linear terms, which are also continuously Fréchet differentiable. The nonlinear terms themselves are continuous bilinear forms and, therefore, continuously Fréchet differentiable. In particular, for the nonlinear term $\Lambda(y)y$ and for any given $y,h \in Y$ 
we can write
\begin{equation}
	\frac 1t (\Lambda(y + th) (y + th) - \Lambda(y) y) = \frac 1t (\Lambda(th) y + \Lambda(y) th + \Lambda(th) th) = \Lambda(h) y + \Lambda(y) h + t \Lambda(h) h
\end{equation}
 Sending $t \to 0$, we obtain the directional derivative at the point  $y$ in the direction $h$. Similar calculations can be carried out for the term $\tilde{K}(u) y$. Hence we can conclude the following Lemma.
\begin{Lem}
    Let $y \in Y$, $u \in U^{\text{ad}}$ be given. Then for every $h \in Y$ and $k \in U$ with $u + \epsilon k \in U^{\text{ad}}$, the directional derivatives \begin{equation}
        e_y(y, u) h = \lim_{t \to 0} \frac 1t \qty(e(y+th, u) - e(y, u)), \quad e_u(y, u) k = \lim_{t \to 0} \frac 1t \qty(e(y, u+tk) - e(y, u))
    \end{equation}
    exist and the following equalities 
    \begin{align}
	e_y(y, u) h &= \qty(\delta h + L h + \Lambda(y) h + \Lambda(h) y + \tilde{K}(u) h - \sigma \laplace h,\, h(t = 0),\, h(a = 0) - \int_0^\amax \beta h \dd{a}),\\
	e_u(y, u) k &= (\tilde{K}(k)y, 0, 0)
\end{align}
hold
\end{Lem}
\begin{Lem}
\label{lem:bounded_inverse}
	Suppose that Assumption \ref{assump:non} holds. For all $u \in U^{\text{ad}}$ the linear map $e_y(y(u), u)$ has a bounded inverse.
\end{Lem}
\begin{proof}
	The statement is equivalent to demonstrating that for every given tuple $(f, h_0, B_0) \in Z$, the equation 
        \begin{gather} \label{eq:linearizedEquation}
		\delta h + (L +\Lambda(y) + \tilde{K}(u)) h + \Lambda(h) y - \sigma \laplace h = f,\\
		h(t = 0) = h_0,\quad h(a = 0) - \int_0^\amax \beta h \dd{a} = B_0,\quad \partial_\nu h(x \in \partial \Omega) = 0.
	\end{gather}
	admits a unique weak solution $h \in Y$. To prove this, similarly to the proof of Theorem \ref{thm:linReg}, we proceed through several steps. First, by neglecting the terms involving $\Lambda$, $K$ and $\beta$,  we consider the following linearized equation
    \begin{gather}
    \label{eq:linearized_eq}
		\delta h + L h - \sigma \laplace h = f,\\
		h(t = 0) = h_0,\quad h(a = 0) = B_0,\quad \partial_\nu h(x \in \partial \Omega) = 0.
	\end{gather}
        Similar to \eqref{eq:Bfixed} for \eqref{eq:linear_with_age}, it can be shown that \eqref{eq:linearized_eq} has a solution of the form
 \begin{equation}
 \label{eq:sol_of_linearized}
		h(t, a, x) = \begin{cases}
			U(a,0) B_0(t-a) + S(a, 0) \restr{f}_{\chak(t-a)} & t > a\\
			U(a,a-t) h_0(a-t) + S(a, a-t) \restr{f}_{\chak(t-a)} & t \leq a
		\end{cases}.
	\end{equation}
	In the next step, we replace $B_0$ by a function $b \in L^2([0, T], H)$  satisfying 
    \begin{equation}
    \label{eq:implict_linearized}
    b(t) = h(a = 0) = B_0(t) + \int_0^\amax \beta h \dd{a}.
    \end{equation}
    Together with \eqref{eq:sol_of_linearized}, and in a similar manner to \eqref{eq:voltera}, we obtain the Volterra equation 
    \begin{align}
		b(t) &= B_0(t) + \int_0^{\min(t, \amax)} \beta(\alpha, x) U(\alpha, 0) b(t - \alpha) \dd{\alpha}\\
		&\quad + \int_{\min(t, \amax)}^\amax \beta(\alpha) U(\alpha, \alpha - t) h_0(\alpha-t) \dd{\alpha} + \int_{t-\amax}^t \beta(t-t_0) S(t-t_0, t_0^-) \restr{f}_{\chak(t_0)} \dd{t_0},
	\end{align}
	which by \cite[Cor. 0.2]{Pruess} has a solution in $L^2((0, T), H)$ satisfying \begin{equation} \label{eq:smallbEstimate}
		\norm{b}_{L^2((0, T), H)} \leqc \norm{h_0}_\hilbert + \norm{B_0}_{L^2((0, T), H)} + \norm{f}_{L^2((0, T) \times \mathcal{I}, V')}.
	\end{equation} Hence, the solution of \eqref{eq:linearized_eq} with initial condition $(a = 0)$ given by \eqref{eq:implict_linearized} can be expressed by
    \begin{equation}
		h(t, a, x) = \begin{cases}
			U(a,0) b(t-a) + S(a, 0) \restr{f}_{\chak(t-a)} & t > a\\
			U(a,a-t) h_0(a-t) + S(a, a-t) \restr{f}_{\chak(t-a)} & t \leq a
		\end{cases},
	\end{equation}
	To include the remaining terms in the equation, we will employ the Banach fixed-point argument, as demonstrated in the proof of Theorem  \ref{Thm:existence_Non}. For a given  $h \in Y$,  let $\Phi(h)$ denote the solution $k \in Y$ to the following equation
    \begin{gather}
		\delta k + L k - \sigma \laplace k = f - (\Lambda(y) + \tilde{K}(u)) h - \Lambda(h) y,\\
		k(t = 0) = h_0,\quad k(a = 0) - \int_0^\amax \beta k \dd{a} = B_0,\quad \partial_\nu k(x \in \partial \Omega) = 0.
	\end{gather}
	With similar arguments as in the proof of Theorem \ref{Thm:existence_Non}, $\Phi$ is well-defined. 
    Further, using \eqref{eq:Lambdaestimate}, we can derive the following estimate 
     \begin{align}
        \MoveEqLeft \norm{\Phi(h) - \Phi(k)}_{L^{\infty}((0,\infty),\hilbert)} \leqc \norm{\Lambda(y) (h-k)}_{L^2((0, T) \times \mathcal{I}, V')} + \norm{\tilde{K}(u)(h-k)}_{L^2((0, T) \times \mathcal{I}, V')}\\
        & + \norm{\Lambda(h-k)y}_{L^2((0, T) \times \mathcal{I}, V')}\leqc \norm{h-k}_{L^2((0, T) \times \mathcal{I}, H)}\leqc \sqrt{T} \norm{h-k}_{L^\infty((0, T), \hilbert))}
	\end{align}
    Hence, choosing $T$ sufficiently small, we can conclude that $\Phi$ is a contraction on $C([0, T], \hilbert))$ (which is a superset of $Y$). Therefore, we have a local existence (in time) of the solution to \eqref{eq:linearizedEquation}. The existence of the global solution follows from the following energy estimate which holds globally. Let $h$ be any fixed point of $\Phi$, then combining \ref{eq:yEstimate} and \ref{eq:smallbEstimate} yields for almost all $t \in (0, T)$ that
    \begin{gather}
        \norm{h(t)}_\hilbert^2 \leqc \norm{f}_{L^2((0, T), V')}^2 + \norm{h_0}_\hilbert^2 + \norm{B_0}_{L^2((0, T), H)}^2 + \norm{\Lambda(y) h + \Lambda(h) y + \tilde{K}(u) h}_{L^2((0, t), \hilbert)}^2\\
        \leqc \norm{f}_{L^2((0, T), V')}^2 + \norm{h_0}_\hilbert^2 + \norm{B_0}_{L^2((0, T), H)}^2 + \qty(\norm{y}_{L^\infty((0, T), \hilbert)}^2 + \norm{u}_{L^\infty((0, T), \hilbert)}^2) \norm{h}_{L^2((0, t), \hilbert)}^2,
    \end{gather}
    where in the last line, we have used \eqref{eq:Lambdaestimate}. Now, an application of Gronwall's lemma yields the following estimate 
    \begin{equation}
        \norm{h}_{L^{\infty}((0,T),\hilbert)}^2 \leqc  e^{\qty(\norm{y}_{L^\infty((0, T), \hilbert)}^2 + \norm{u}_{L^\infty((0, T), \hilbert)}^2)T} \qty(\norm{f}_{L^2((0, T), V')}^2 + \norm{h_0}_\hilbert^2 + \norm{B_0}_{L^2((0, T), H)}^2).     
    \end{equation}
    This uniform estimate, combined with standard continuation arguments, shows that the solution obtained via Banach's Fixed Point Theorem can indeed be extended to the entire time interval $(0, T)$. Furthermore, the estimate establishes the uniqueness of the solution, as it directly shows that the difference between two solutions to the linear equation with identical initial and boundary conditions must be zero at all times. This completes the proof.
\end{proof}

Next, we will derive the adjoint equation. In other words, we will find $p \in Z'$ with  \begin{equation}
	Z' = L^2((0, T) \times \mathcal{I}, V) \times L^2(\mathcal{I}, H) \times L^2((0, T), H)
\end{equation}
which satisfies
\begin{equation}
    e_y(y(u), u)^* p = -J_y(y(u), u).
\end{equation}
 To show this,  we can write for all $h \in Y$ that
 \begin{align}
	-\skpr{g y(u), gh} &= \langle p_1, \delta h + L h + \Lambda(y) h + \Lambda(h) y + \tilde{K}(u) h - \sigma \laplace h \rangle_{L^2((0, T) \times \mathcal{I}, V), L^2((0, T) \times \mathcal{I}, V')}\\
	&\quad + \skpr{p_2, h(t = 0)}_{L^2(\mathcal{I}, H)} + \skpr{p_3, h(a = 0) - \int_0^\amax \beta h \dd{a}}_{L^2((0, T), H)}\\
	&\leftstackrel{PI}{=} - \skpr{h, \delta p_1}_{L^2((0, T) \times \mathcal{I} \times \Omega)} + \skpr{h(t = T), p_1(t = T)}_{L^2(\mathcal{I} \times \Omega)} - \skpr{h(t = 0), p_1(t = 0)}_{L^2(\mathcal{I} \times \Omega)}\\
	&\quad + \skpr{h(a = \amax), p_1(a = \amax)}_{L^2((0, T) \times \Omega)} - \skpr{h(a = 0), p_1(a = 0)}_{L^2((0, T) \times \Omega)}\\
	&\quad + \skpr{h, L\tran p_1} + \skpr{\Lambda(a, x, h) y + \Lambda(a, x, y) h + \tilde{K}(u) h, p_1}_{L^2((0, T) \times \mathcal{I} \times \Omega)}\\
	&\quad - \skpr{h, \sigma(a) \laplace p_1}_{L^2((0, T) \times \mathcal{I} \times \Omega)} - \skpr{\partial_\nu h, p_1}_{L^2((0, T) \times \mathcal{I} \times \partial \Omega)} + \skpr{h, \partial_\nu p_1}_{L^2((0, T) \times \mathcal{I} \times \partial \Omega)}\\
	&\quad + \skpr{h(t = 0) - h_0, p_2}_{L^2(\mathcal{I} \times \Omega)} + \skpr{h(a = 0) - \int_0^\amax \beta(\alpha, x) h(t, \alpha, x) \dd{\alpha}, p_3}_{L^2((0, T) \times \Omega)}.
\end{align}
Recalling  $\Lambda(h)^{\beta \gamma}(t, a, x) = \int_0^\amax \int_\Omega k^{\beta \gamma \delta}(a, x, \alpha, \xi) h_\delta(t, \alpha, \xi) \dd{\xi} \dd{\alpha}$  and setting $z \coloneqq (a, x)$, $\zeta \coloneqq (\alpha, \xi)$, we can rewrite the term with $\Lambda(h) y$ as \begin{align}
	\skpr{\Lambda(h) y, p_1} &= \int_0^T \int \int k^{\beta \gamma \delta}(z, \zeta) h_\delta(t, \zeta) \dd{\zeta} y_\gamma(t, z) p_{1 \beta}(t, z) \dd{z} \dd{t}\\
	&= \int_0^T \int h_\delta(t, \zeta) \int k^{\beta \gamma \delta}(z, \zeta)  y_\gamma(t, z) p_{1 \beta}(t, z) \dd{z} \dd{\zeta} \dd{t}\\
	&\eqqcolon \skpr{h, \tilde{\Lambda}_y(p_1)}
\end{align}
where $\tilde{\Lambda}_y$ is a nonlocal linear operator. Further, the term with the birth condition can be rewritten as \begin{align}
	\skpr{\int_0^\amax \beta(\alpha) h(\alpha) \dd{\alpha}, p_3}_{L^2((0, T) \times \Omega)} &= \iiint \beta(\alpha, \xi) h(\theta, \alpha, \xi) \cdot p_3(\theta, \xi) \dd{(\theta, \alpha, \xi)}\\
	&= \skpr{h, \beta\tran p_3}_{L^2((0, T) \times \mathcal{I} \times \Omega)}
\end{align}
Therefore, we can write
\begin{align}
	-\skpr{g y(u), gh} &= - \skpr{\delta p_1, h} + \skpr{p_1(t = T), h(t = T)}_{L^2(\mathcal{I} \times \Omega)} - \skpr{p_1(t = 0), h(t = 0)}_{L^2(\mathcal{I} \times \Omega)}\\
	&\quad + \skpr{p_1(a = \amax), h(a = \amax)}_{L^2((0, T) \times \Omega)} - \skpr{p_1(a = 0), h(a = 0)}_{L^2((0, T) \times \Omega)}\\
	&\quad + \langle (L + \Lambda(y(u)) + \tilde{K}(u))\tran p_1+ \tilde{\Lambda}_{y(u)}(p_1)- \sigma \laplace p, h \rangle_{L^2((0, T) \times \mathcal{I}, V'), L^2((0, T) \times \mathcal{I}, V)} \\
	&\quad + \skpr{p_2, h(t = 0)}_{L^2(\mathcal{I}, H)} + \skpr{p_3, h(a = 0)}_{L^2((0, T), H)} - \skpr{\beta\tran p_3, h}_{L^2((0, T) \times \mathcal{I} \times \Omega)},
\end{align}
which represents the weak formulation of the adjoint equation. Setting $p \coloneqq p_1$ and using the fact that $p_2 = p(t = 0)$ and $p_3 = p(a = 0)$, we deduce that 
\begin{gather}
\label{eq:adoint_equation}
	- \delta p  + (L + \Lambda(y(u)) + \tilde{K}(u))\tran p + \tilde{\Lambda}_{y(u)}(p) - \sigma \laplace p - \beta(a, x)\tran p(a = 0) = -g\tran g y(u),\\
	p(t = T) = 0, \quad p(a = \amax) = 0,\quad \partial_\nu p(x \in \partial \Omega) = 0,
\end{gather}
which is the adjoint equation. Thus, we have completed the derivation of the adjoint equation.

\begin{Bem}    
Interestingly, a term of the form $\beta(a, x)\tran p(a = 0)$ naturally arises in the equation due to the implicit boundary condition. Notably, a similar term, $\beta\tran p(a = 0)$, appears in the adjoint equation for a class of optimal control problems governed by age-structured models without spatial variable $x$ (ODEs), as shown in \cite[Thm. 4.11]{AnitaArnautuCapasso}.
\end{Bem}

\begin{Thm}Suppose that Assumption \ref{assump:non} holds. Then the adjoint equation \eqref{eq:adoint_equation} has a unique solution in $Y$.
\end{Thm}
\begin{proof}
	Setting  $h(t, a) \coloneqq p(T-t, \amax-a)$, the adjoint equation can be rewritten as \begin{gather}
    \begin{multlined}
        \delta h + (L +\Lambda(y) + \tilde{K}(u))\tran(T - t, \amax - a) h + \tilde{\Lambda}_{y(T - t, \amax - a)}(h) - \sigma(\amax - a) \laplace h\\
        - \beta(\amax - a)\tran h(a = \amax) = -g\tran g y(T - t, \amax - a),
    \end{multlined}	\\
		h(t = 0) = 0,\quad h(a = 0) = 0,\quad \partial_\nu h(x \in \partial \Omega) = 0.
	\end{gather}
Comparing this equation with \eqref{eq:linearizedEquation}, we observe that they share a similar form, except for the initial conditions at $a = 0$ and $t =0$, which in this case are zero, and a term $\int_0^\amax \beta h \dd{a}$ on the boundary which leads to the term $\beta (\amax - a)\tran h(t, \amax, x)$ in \eqref{eq:adoint_equation}. Therefore, to prove the theorem, we can apply similar arguments to those in the proof of Lemma \ref{lem:bounded_inverse}, with only slight adaptations. More specifically, we can easily obtain a solution $h$ of the linearized equation
\begin{gather}
    \begin{multlined}
        \delta h + L\tran(T - t, \amax - a) h - \sigma(\amax - a) \laplace h - f = -g\tran g y(T - t, \amax - a),
    \end{multlined}	\\
		h(t = 0) = 0,\quad h(a = 0) = 0,\quad \partial_\nu h(x \in \partial \Omega) = 0,
	\end{gather}
as described in the proof of Corollary \ref{cor:WeakSolution}. Similarly to Lemma \ref{lem:well_linear}, we obtain evolution operators $\tilde{U}$ and $\tilde{S}$ for the equation along characteristics, which we can use to represent the solution as 
\begin{equation} \label{eq:AdjRepresentation}
		h(t, a, x) = \begin{cases}
			\tilde{S}(a, 0) \restr{f}_{\chak(t-a)} & t > a\\
			\tilde{S}(a, a-t) \restr{f}_{\chak(t-a)} & t \leq a
		\end{cases}.
	\end{equation}

The next step is to replace $f$ in this equation with $f + \beta\tran b$, where   $b \in L^2((0, T), H)$ is such that $b = h(a = \amax)$. Substituting this into \eqref{eq:AdjRepresentation} and applying Duhamel's principle from Remark \ref{Duhamel}, we obtain the Volterra equation
\begin{align}
		b(t) &= \tilde{S}(\amax, (\amax - t)^+) \restr{f}_{\chak(t)-\amax} + \int_{(\amax - t)^+}^\amax \tilde{U}(\amax, r)  \beta(\amax-r)\tran b(t-\amax+r) \dd{r}\\
		&= \tilde{S}(\amax, (\amax - t)^+) \restr{f}_{\chak(t)-\amax} + \int_0^{\min(t, \amax)} \tilde{U}(\amax, \amax - s)  \beta(s)\tran b(t-s) \dd{s}.
	\end{align}
	A solution $b \in L^2((0, T), H)$ can then be found as in Theorem \ref{Thm:Bexistence}. The remainder of the proof, specifically the inclusion of the missing terms, follows similarly to the proof of Lemma ~\ref{lem:bounded_inverse}.
\end{proof}

The first-order necessary optimality conditions are then given by the following result.

\begin{Kor} 
Suppose that Assumption \ref{assump:non} holds, and let the pair $(y^*, u^*) \in Y \times U^{ad}$ be a solution to \eqref{Opt}. Then, this pair satisfies the following variational inequality
\begin{equation}
\begin{split}
		&\alpha (u^*, u - u^*)_{L^2((0,T),\R^{M \times N})} + (p^*, K(u - u^*) y^*)_{L^2((0, T), \hilbert)} \\ & = (\alpha u^* + \tilde{K}^*\left( \operatorname{diag}(y^* \odot p^*)\right), u - u^*)_{L^2((0,T),\R^{M \times N})} \geq 0  \quad  \text{ for all }  u \in U^{ad},
\end{split}
 	\end{equation}
    where $\tilde{K}^*$ is the adjoint operator of $\tilde{K}$, $y^* = y(u^*)$ is the solution to \eqref{eq:StateEquation} corresponding to $u^*$, and $p^*$ denotes the solution of the adjoint equation \eqref{eq:adoint_equation} associated with $u^*$ and $y^*$. The term $\operatorname{diag}(y^* \odot p^*)$ represents a diagonal matrix with entries  $y^*_1 \cdot p^*_1, \ldots, y^*_n \cdot p^*_n$ on the main diagonal and zeros elsewhere.   
\end{Kor}
\begin{proof}
    The proof follows directly from \cite[Cor. 1.3]{HPUU}.
\end{proof}

\section{Numerical Experiments}
This section reports on the numerical implementation of the optimal control problem \eqref{Opt} governed by \eqref{eq:svir} model. The spatial domain was chosen as the interval $\Omega =[0,1]$. We applied a standard finite difference scheme for spatial discretization with some step size $\Delta x$. The nonlocal integral operator $\Lambda(\cdot) $ was approximated using the trapezoidal rule, applied twice, once for integration over the spatial domain and once over the age interval $\mathcal{I}$. The birth law was approximated similarly. Then, following the approach used in our existence results for \eqref{eq:linear_with_age}, we computed the numerical solution by restricting the equation to characteristic lines, transforming it into an ordinary differential equation of the form
\begin{equation}
\delta y = F(t, a, y).
\end{equation}
To tackle this, we applied the Crank–Nicolson method for temporal/age discretization, which yields the following scheme
\begin{equation}
		\frac{y(t + \Delta t, a + \Delta t) - y(t, a)}{\Delta t} = \frac 12 F(t, a, y(t, a)) + \frac 12 F(t + \Delta t, a + \Delta t, y(t + \Delta t, a + \Delta t)),
\end{equation} 
 where $\Delta t$ denotes the temporal step size. The nonlinear terms were treated explicitly using the Adams-Bashforth method; see, e.g., \cite{HNW93}. Notably, the same step size was chosen for both time and age discretization. This choice, while convenient, poses challenges for real-world simulations: typically, $\Delta t$ is on the order of days or less, while the maximum age $\amax$ in $I$ can span several decades. As a result, the age variable requires very fine discretization, potentially leading to the curse of dimensionality.

To solve problem \eqref{Opt}, we followed the discretize-then-optimize approach, employing the projected gradient method for the associated reduced problem defined as
\begin{equation}
 \min_{u \in U^{ad}} \mathcal{J}(u) \coloneqq \min_{u \in U^{ad}} J(u,y(u)),
\end{equation}
where $y(u)$ is the unique solution to \eqref{eq:StateEquation} corresponding to the control $u$. More precisely,  we used the iterative update rule
 \begin{equation}
u^{k+1} = P_{U^{ad}}(u^k -\alpha_k\mathcal{J}'(u^k) ),   \quad \text{ for  } k\geq 0 
\end{equation}
where $P_{U^{ad}}$ stands for the orthogonal projection into $U^{ad}$ and $\mathcal{J}'$ denotes the gradient of the reduced problem and the step size $\alpha_k$ was determined using a non-monotone line search algorithm \cite{AB23} which uses the Barziali-Borwein step sizes \cite{Azmi, BB88} corresponding to $\mathcal{J}$ as the initial trial step size.  The optimization algorithm was terminated when the norm of the difference between two successive iterations, divided by the norm of the previous iteration, was less than $10^{-8}$.

\begin{table}[h]
    \centering
    \begin{tabular}{c|c|c}
    Parameter & Description & Value\\\hline
    $T$ & Maximal time & 5\\
    $\amax$ & Maximum age & 1\\
    $\alpha$ & Control cost parameter & 500\\
    $c$ & Loss of vaccine immunity & 0.18564\\
    $\mu$ & Natural death rate & $e^{-a} \cdot a^5$\\
    $\phi_1$ & Vaccine protection from infection & 0.0052\\
    $\phi_2$ & Recovery protection from infection & 0.00062\\
    $\delta$ & Infection death rate & 0.0018\\ 
    $\gamma$ & Recovery rate & 0.278574\\
    $\lambda$ & Infection rate kernel & $(0.1 - \abs{x - \xi})^+$\\
    $\beta$ & Birth rate & $\frac{6.78}{\amax} a^2 (\amax - a) (1 + \sin(\pi \frac{a}{\amax}))$\\
    $\sigma_S$ & Susceptible diffusion coefficient & $0.1 e^{-0.1 a}$ \\
    $\sigma_V$ & Vaccinated diffusion coefficient & $0.1 e^{-0.1 a}$ \\
    $\sigma_I$ & Infective diffusion coefficient & $0.05 e^{-0.1 a}$ \\
    $\sigma_R$ & Recovered diffusion coefficient & $0.1 e^{-0.1 a}$ \\
    \end{tabular}
    \caption{Parameter Setting}
    \label{tab:values}
\end{table}
Throughout the numerical simulation, we used the parameters listed in Table \ref{tab:values}. Most of these parameters are taken from \cite{AbRaSch} and adjusted for our purposes. The birth and death rates are adopted from \cite[p. 155]{AnitaArnautuCapasso}. As initial conditions, we assume a population uniformly distributed across age and space, consisting of 1000 susceptible and 10 infectious individuals. The control is assumed to act only in the central region of the domain, $\Omega_1 = (0.45, 0.55)$, and uniformly across the age intervals defined by $a_0 = 0$, $a_1 = 0.18$, $a_2 = 0.3$, $a_3 = 0.5$, $a_4 = 0.7$, and $a_5 = 1$. Consequently, we set $N = 5$ and $M = 1$.

\begin{Example}
	\label{Exp:1}
	In this example, we set $\Delta x = 0.01$, $\Delta t = 0.005$, and $\bar{u}=10$. The structure of the optimal control for different age classes is illustrated in Figure \ref{fig:HugeStructure}, where each strip corresponds to an age class. As shown, it is preferable to administer the vaccine during the early stages of the epidemic, and in later stages, prioritize younger age groups over older ones.
	
	The evolution of the total number of infectious individuals, calculated by \begin{equation}
		\int_0^\amax \int_\Omega I(t, a, x) \dd{x} \dd{a}
	\end{equation} for both the controlled and uncontrolled ($u = 0$) cases,  is depicted in Figure \ref{fig:HugeNumber}. It can be observed that control measures or vaccination effectively reduce the number of infectious individuals over time, as desired. Figures \ref{fig:HugeControlled} and \ref{fig:HugeUncontrolled} present snapshots of the four compartments at the final time for the controlled and uncontrolled cases, respectively. Comparing these figures, a visible dent appears in the optimal state $I$. Additionally, the graph of the number of infected individuals shows a bump in the middle, reflecting the fact that the optimal control strategy favors vaccinating younger individuals.
	\end{Example}

\begin{figure}[H]
\centering
\begin{subfigure}{.8\textwidth}
    \includegraphics[width=\linewidth]{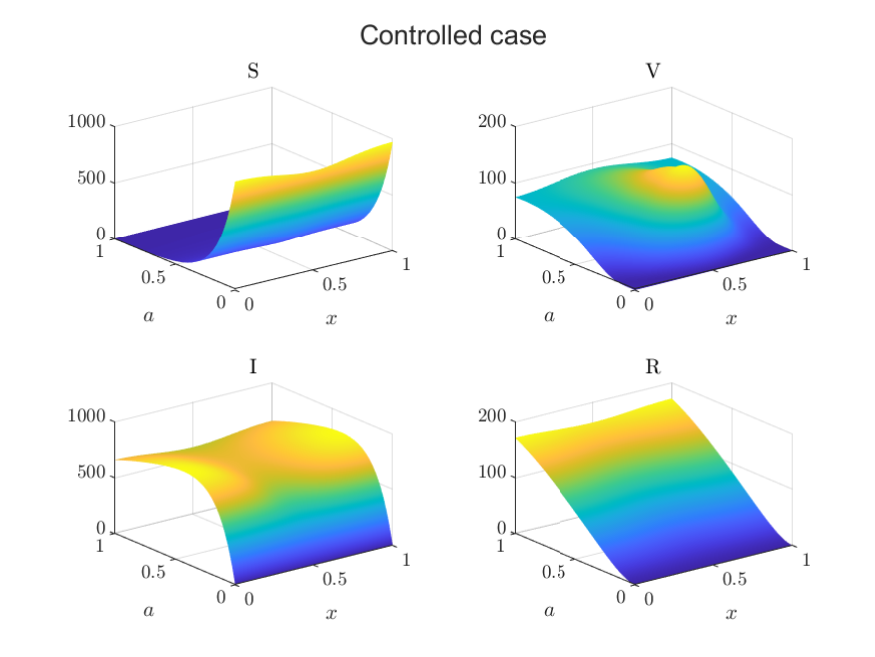}
    \caption{Controlled state at the final time}
    \label{fig:HugeControlled}
\end{subfigure}
\begin{subfigure}{.8\textwidth}
    \includegraphics[width=\linewidth]{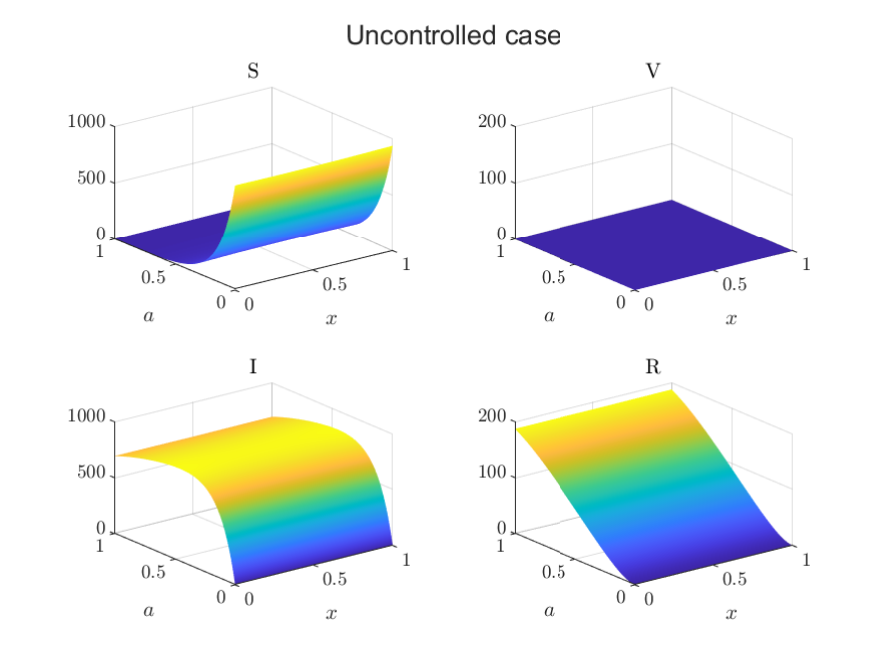}
    \caption{Uncontrolled state at the final time}
    \label{fig:HugeUncontrolled}
\end{subfigure}
\caption{Numerical results of Example \ref{Exp:1}}
\label{fig:HugeResults}
\end{figure}

\begin{figure}[H] \ContinuedFloat
	\centering
	\begin{subfigure}{.45\textwidth}
		\includegraphics[width=\linewidth]{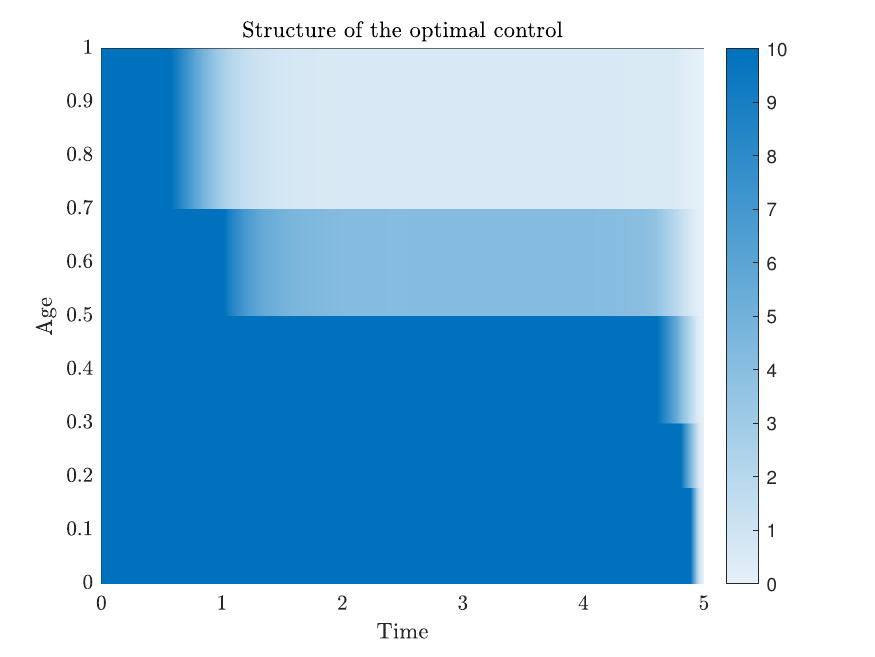}
		\caption{Optimal control structure, with darker blue indicating higher values of $u$.}
		\label{fig:HugeStructure}
	\end{subfigure}
	\begin{subfigure}{.45\textwidth}
		\includegraphics[width=\linewidth]{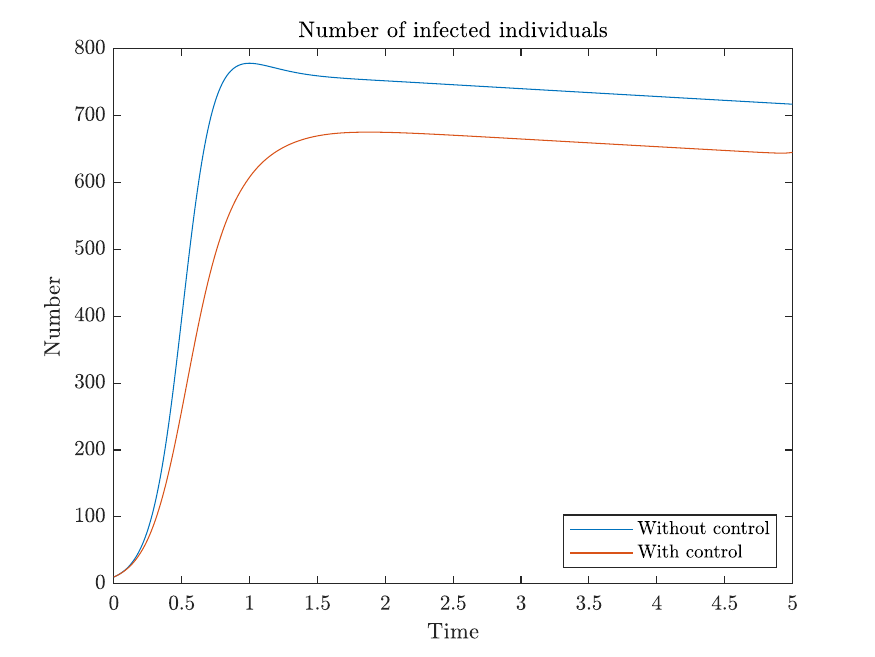}
		\caption{Total number of infectious individuals: blue — controlled, red — uncontrolled.}
		\label{fig:HugeNumber}
	\end{subfigure}
\end{figure}

\FloatBarrier

\begin{Example}
\label{Exp:2}
It is insightful to vary $\bar{u}$ and observe how it impacts the optimal cost and state, as a higher control bound is more effective but also more expensive. In this example, we explore this trade-off. To speed up computations, we used a coarser grid with $\Delta x = 0.1$ and $\Delta t = 0.01$, and computed the optimal control for $\bar{u} \in \{10, 20, 50, 80\}$. The corresponding results are presented in Figure \ref{fig:VarUmaxResults} and Table \ref{tab:VarUmaxTarget}.

Comparing the top-left plot in Figure \ref{fig:VarUmaxControl} with Figure \ref{fig:HugeStructure}, we observe that the results are consistent, and different mesh sizes do not significantly influence the outcomes.

From Table \ref{tab:VarUmaxTarget}, we see that as $\bar{u}$ increases, the optimal cost functional decreases, indicating that maintaining a high vaccination rate can be beneficial despite the associated costs. However, for $\bar{u} = 80$, the control reaches a maximum value of approximately $62.2$, suggesting that beyond a certain point, increasing $\bar{u}$ provides no additional benefit. Most control values lie below $50$, which explains why the optimal controls for $\bar{u} = 80$ and $\bar{u} = 50$ are nearly identical. In fact, the curves representing the total number of infected individuals for these two cases in Figure \ref{fig:VarUmaxNumber} completely overlap.
\end{Example}

\begin{table}[h]
    \centering
    \begin{tabular}{r|r}
        $\bar{u}$ & $J(y(\bar{u}), \bar{u})$ \\\hline
          0 & 1111450\\
         10 &  896491\\
         20 &  822924\\
         50 &  787393\\
         80, $\infty$ & 787382 
    \end{tabular}
    \caption{Optimal value of the cost  functional for different values of $\bar{u}$ from Example \ref{Exp:2}}
    \label{tab:VarUmaxTarget}
\end{table}

\begin{figure}[h]
\centering
\begin{subfigure}{0.8\linewidth}
    \includegraphics[width=\linewidth]{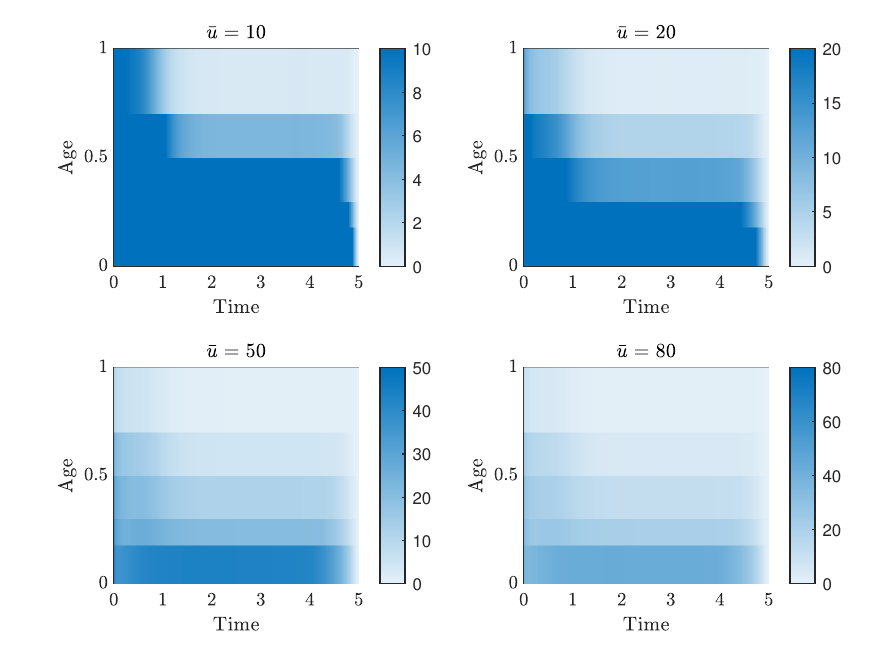}
    \caption{Optimal control structure for different values of $\bar{u}$}
    \label{fig:VarUmaxControl}
\end{subfigure}
\end{figure}
\begin{figure}[h]\ContinuedFloat
\centering
\begin{subfigure}{0.8\linewidth}
    \includegraphics[width=\linewidth]{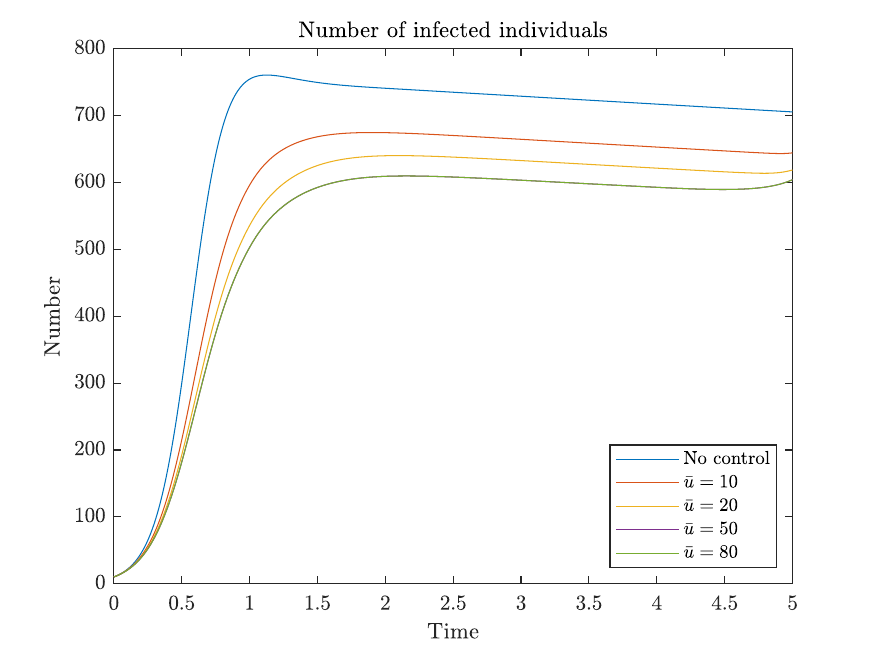}
    \caption{Total number of infectious individuals for different values of $\bar{u}$}
    \label{fig:VarUmaxNumber}
\end{subfigure}
\caption{Comparison of different values for $\bar{u}$ of Example \ref{Exp:2}}
\label{fig:VarUmaxResults}
\end{figure}

\section*{Conclusions}

In this work, we analyzed a class of nonlinear epidemic models incorporating age and spatial structures, with a nonlocal infection term dependent on age, space, and potentially time. We established the well-posedness of the state partial differential equation. By introducing the vaccination rate as a control parameter, we addressed the existence of an optimal control under suitable conditions and characterized it through first-order optimality conditions. Numerical examples illustrated the behavior of the optimal control strategies, highlighting the practical significance of the theoretical results.

From a practical perspective, exploring infinite-horizon optimal control for age-structured nonlinear epidemic models would be a valuable future direction. Another promising avenue for future research is the design of feedback control laws that dynamically adjust based on the system state, enabling real-time steering of the epidemic dynamics. In managing pandemics, long-term control strategies and robust mechanisms that adapt to perturbations and evolve over time are often essential. One effective framework is receding horizon control (RHC), also known as model predictive control (MPC), which approximates the solution to an infinite-horizon problem through a sequence of finite-horizon ones, continuously refining the control strategy via feedback. However, ensuring the stabilizability of this control remains a significant challenge that requires further investigation.\\

\textbf{Acknowledgements} The authors appreciate and acknowledge Prof. Reinhard Racke for his helpful comments and insights on this manuscript.

\bibliographystyle{abbrvurl}
\bibliography{bibliography}

\begin{thebibliography}{10}

\bibitem{AbRaSch}
H.~Abboubakar, R.~Racke, and N.~Schlosser.
\newblock A reaction-diffusion model for the transmission dynamics of the
  coronavirus pandemic with reinfection and vaccination process.
\newblock {\em Konstanzer Schriften in Mathematik}, (409), 2023.

\bibitem{AdamsFournier}
R.~A. Adams and J.~J.~F. Fournier.
\newblock {\em Sobolev spaces}, volume 140 of {\em Pure and Applied Mathematics
  (Amsterdam)}.
\newblock Elsevier/Academic Press, Amsterdam, second edition, 2003.

\bibitem{AnitaArnautuCapasso}
S.~Ani\c{t}a, V.~Arn\u{a}utu, and V.~Capasso.
\newblock {\em An introduction to optimal control problems in life sciences and
  economics}.
\newblock Modeling and Simulation in Science, Engineering and Technology.
  Birkh\"auser/Springer, New York, 2011.
\newblock From mathematical models to numerical simulation with
  MATLAB$^\circledR$.
\newblock \href {https://doi.org/10.1007/978-0-8176-8098-5}
  {\path{doi:10.1007/978-0-8176-8098-5}}.

\bibitem{ACGRR}
F.~Auricchio, P.~Colli, G.~Gilardi, A.~Reali, and E.~Rocca.
\newblock Well-posedness for a diffusion–reaction compartmental model
  simulating the spread of covid-19.
\newblock {\em Mathematical Methods in the Applied Sciences},
  46(12):12529--12548, 2023.
\newblock URL: \url{https://onlinelibrary.wiley.com/doi/abs/10.1002/mma.9196},
  \href
  {https://arxiv.org/abs/https://onlinelibrary.wiley.com/doi/pdf/10.1002/mma.9196}
  {\path{arXiv:https://onlinelibrary.wiley.com/doi/pdf/10.1002/mma.9196}},
  \href {https://doi.org/10.1002/mma.9196} {\path{doi:10.1002/mma.9196}}.

\bibitem{AB23}
B.~Azmi and M.~Bernreuther.
\newblock On the nonmonotone linesearch for a class of infinite-dimensional
  nonsmooth problems.
\newblock 2023.
\newblock submitted.
\newblock URL: \url{https://arxiv.org/abs/2303.01878}.

\bibitem{Azmi}
B.~Azmi and K.~Kunisch.
\newblock On the convergence and mesh-independent property of the
  barzilai–borwein method for pde-constrained optimization.
\newblock {\em IMA Journal of Numerical Analysis}, 42(4):2984--3021, 08 2021.
\newblock \href
  {https://arxiv.org/abs/https://academic.oup.com/imajna/article-pdf/42/4/2984/46323811/drab056.pdf}
  {\path{arXiv:https://academic.oup.com/imajna/article-pdf/42/4/2984/46323811/drab056.pdf}},
  \href {https://doi.org/10.1093/imanum/drab056}
  {\path{doi:10.1093/imanum/drab056}}.

\bibitem{BB88}
J.~Barzilai and J.~M. Borwein.
\newblock Two-point step size gradient methods.
\newblock {\em IMA J. Numer. Anal.}, 8(1):141--148, 1988.
\newblock \href {https://doi.org/10.1093/imanum/8.1.141}
  {\path{doi:10.1093/imanum/8.1.141}}.

\bibitem{BDKW}
S.~Bentout, S.~Djilali, T.~Kuniya, and J.~Wang.
\newblock Mathematical analysis of a vaccination epidemic model with nonlocal
  diffusion.
\newblock {\em Math. Methods Appl. Sci.}, 46(9):10970--10994, 2023.
\newblock \href {https://doi.org/10.1002/mma.9162}
  {\path{doi:10.1002/mma.9162}}.

\bibitem{Bongarti}
M.~Bongarti, C.~Parkinson, and W.~Wang.
\newblock Optimal control of a reaction-diffusion epidemic model with
  noncompliance, 2024.
\newblock URL: \url{https://arxiv.org/abs/2407.17298}, \href
  {https://arxiv.org/abs/2407.17298} {\path{arXiv:2407.17298}}.

\bibitem{BDJ}
F.~Brauer, P.~van~den Driessche, and J.~Wu, editors.
\newblock {\em Mathematical epidemiology}, volume 1945 of {\em Lecture Notes in
  Mathematics}.
\newblock Springer-Verlag, Berlin, 2008.
\newblock Mathematical Biosciences Subseries.
\newblock \href {https://doi.org/10.1007/978-3-540-78911-6}
  {\path{doi:10.1007/978-3-540-78911-6}}.

\bibitem{BredaReggiVermiglio}
D.~Breda, S.~De~Reggi, and R.~Vermiglio.
\newblock A numerical method for the stability analysis of linear
  age-structured models with nonlocal diffusion.
\newblock {\em SIAM J. Sci. Comput.}, 46(2):A953--A973, 2024.
\newblock \href {https://doi.org/10.1137/23M1568971}
  {\path{doi:10.1137/23M1568971}}.

\bibitem{CGMR}
P.~Colli, G.~Gilardi, G.~Marinoschi, and E.~Rocca.
\newblock Optimal control of a reaction–diffusion model related to the spread
  of covid-19.
\newblock {\em Analysis and Applications}, 22(01):111--136, 2024.
\newblock \href
  {https://arxiv.org/abs/https://doi.org/10.1142/S0219530523500197}
  {\path{arXiv:https://doi.org/10.1142/S0219530523500197}}, \href
  {https://doi.org/10.1142/S0219530523500197}
  {\path{doi:10.1142/S0219530523500197}}.

\bibitem{Colombo}
R.~M. Colombo, M.~Garavello, F.~Marcellini, and E.~Rossi.
\newblock An age and space structured sir model describing the covid-19
  pandemic.
\newblock {\em Journal of Mathematics in Industry}, 10(1):22, Aug 2020.
\newblock \href {https://doi.org/10.1186/s13362-020-00090-4}
  {\path{doi:10.1186/s13362-020-00090-4}}.

\bibitem{DautrayLions}
R.~Dautray and J.-L. Lions.
\newblock {\em Mathematical analysis and numerical methods for science and
  technology. {V}ol. 5}.
\newblock Springer-Verlag, Berlin, 1992.
\newblock Evolution problems. I, With the collaboration of Michel Artola,
  Michel Cessenat and H\'{e}l\`ene Lanchon, Translated from the French by Alan
  Craig.
\newblock \href {https://doi.org/10.1007/978-3-642-58090-1}
  {\path{doi:10.1007/978-3-642-58090-1}}.

\bibitem{Fragnelli}
G.~Fragnelli.
\newblock An age-dependent population equation with diffusion and delayed birth
  process.
\newblock {\em International Journal of Mathematics and Mathematical Sciences},
  2005(20):409340, 2005.
\newblock URL:
  \url{https://onlinelibrary.wiley.com/doi/abs/10.1155/IJMMS.2005.3273}, \href
  {https://arxiv.org/abs/https://onlinelibrary.wiley.com/doi/pdf/10.1155/IJMMS.2005.3273}
  {\path{arXiv:https://onlinelibrary.wiley.com/doi/pdf/10.1155/IJMMS.2005.3273}},
  \href {https://doi.org/10.1155/IJMMS.2005.3273}
  {\path{doi:10.1155/IJMMS.2005.3273}}.

\bibitem{HNW93}
E.~Hairer, S.~P. N\o~rsett, and G.~Wanner.
\newblock {\em Solving ordinary differential equations. {I}}, volume~8 of {\em
  Springer Series in Computational Mathematics}.
\newblock Springer-Verlag, Berlin, second edition, 1993.
\newblock Nonstiff problems.

\bibitem{HPUU}
M.~Hinze, R.~Pinnau, M.~Ulbrich, and S.~Ulbrich.
\newblock {\em Optimization with {PDE} constraints}, volume~23 of {\em
  Mathematical Modelling: Theory and Applications}.
\newblock Springer, New York, 2009.

\bibitem{HNVW}
T.~Hyt\"onen, J.~van Neerven, M.~Veraar, and L.~Weis.
\newblock {\em Analysis in {B}anach spaces. {V}ol. {I}. {M}artingales and
  {L}ittlewood-{P}aley theory}, volume~63 of {\em Ergebnisse der Mathematik und
  ihrer Grenzgebiete. 3. Folge. A Series of Modern Surveys in Mathematics
  [Results in Mathematics and Related Areas. 3rd Series. A Series of Modern
  Surveys in Mathematics]}.
\newblock Springer, Cham, 2016.

\bibitem{AnalysisInBanach}
T.~Hyt{\"o}nen, J.~{van Neerven}, M.~Veraar, and L.~Weis.
\newblock {\em Analysis in Banach Spaces: Volume I: Martingales and
  Littlewood-Paley Theory}.
\newblock Ergebnisse der Mathematik und ihrer Grenzgebiete. 3. Folge. Springer,
  2016.
\newblock \href {https://doi.org/10.1007/978-3-319-48520-1}
  {\path{doi:10.1007/978-3-319-48520-1}}.

\bibitem{KangRuan}
H.~Kang and S.~Ruan.
\newblock Mathematical analysis on an age-structured {SIS} epidemic model with
  nonlocal diffusion.
\newblock {\em J. Math. Biol.}, 83(1):Paper No. 5, 30, 2021.
\newblock \href {https://doi.org/10.1007/s00285-021-01634-x}
  {\path{doi:10.1007/s00285-021-01634-x}}.

\bibitem{KermackMcKendrick}
W.~Kermack and A.~McKendrick.
\newblock Contributions to the mathematical theory of epidemics—i.
\newblock {\em Bltn Mathcal Biology 53, 33–55}, 1991.
\newblock \href {https://doi.org/10.1007/BF02464423}
  {\path{doi:10.1007/BF02464423}}.

\bibitem{Kuniya}
T.~Kuniya and R.~Oizumi.
\newblock Existence result for an age-structured sis epidemic model with
  spatial diffusion.
\newblock {\em Nonlinear Analysis: Real World Applications}, 23:196--208, 2015.
\newblock URL:
  \url{https://www.sciencedirect.com/science/article/pii/S1468121814001448},
  \href {https://doi.org/10.1016/j.nonrwa.2014.10.006}
  {\path{doi:10.1016/j.nonrwa.2014.10.006}}.

\bibitem{LiYaMa}
X.-Z. Li, J.~Yang, and M.~Martcheva.
\newblock {\em Age structured epidemic modeling}, volume~52 of {\em
  Interdisciplinary Applied Mathematics}.
\newblock Springer, Cham, [2020] \copyright 2020.
\newblock \href {https://doi.org/10.1007/978-3-030-42496-1}
  {\path{doi:10.1007/978-3-030-42496-1}}.

\bibitem{Perthame}
B.~Perthame.
\newblock {\em Parabolic Equations in Biology: Growth, reaction, movement and
  diffusion}.
\newblock Lecture Notes on Mathematical Modelling in the Life Sciences.
  Springer International Publishing, 2015.
\newblock URL: \url{https://books.google.de/books?id=0pOKCgAAQBAJ}.

\bibitem{Pruess}
J.~Pr\"uss.
\newblock {\em Evolutionary integral equations and applications}, volume~87 of
  {\em Monographs in Mathematics}.
\newblock Birkh\"auser Verlag, Basel, 1993.
\newblock \href {https://doi.org/10.1007/978-3-0348-8570-6}
  {\path{doi:10.1007/978-3-0348-8570-6}}.

\bibitem{RenardyRogers}
M.~Renardy and R.~C. Rogers.
\newblock {\em An introduction to partial differential equations}, volume~13 of
  {\em Texts in Applied Mathematics}.
\newblock Springer-Verlag, New York, second edition, 2004.

\bibitem{SharomiMalik}
O.~Sharomi and T.~Malik.
\newblock Optimal control in epidemiology.
\newblock {\em Annals of Operations Research}, 251(1):55--71, Apr 2017.
\newblock \href {https://doi.org/10.1007/s10479-015-1834-4}
  {\path{doi:10.1007/s10479-015-1834-4}}.

\bibitem{Smoller}
J.~Smoller.
\newblock {\em Shock waves and reaction-diffusion equations}, volume 258 of
  {\em Grundlehren der Mathematischen Wissenschaften}.
\newblock Springer-Verlag, New York-Berlin, 1983.

\bibitem{Walker}
C.~Walker.
\newblock Some remarks on the asymptotic behavior of the semigroup associated
  with age-structured diffusive populations.
\newblock {\em Monatsh. Math.}, 170(3-4):481--501, 2013.
\newblock \href {https://doi.org/10.1007/s00605-012-0428-3}
  {\path{doi:10.1007/s00605-012-0428-3}}.

\bibitem{Walker2}
C.~Walker.
\newblock Well-posedness and stability analysis of an epidemic model with
  infection age and spatial diffusion.
\newblock {\em J. Math. Biol.}, 87(3):Paper No. 52, 46, 2023.
\newblock \href {https://doi.org/10.1007/s00285-023-01980-y}
  {\path{doi:10.1007/s00285-023-01980-y}}.

\bibitem{WangZhangKuniya}
J.~Wang, R.~Zhang, and T.~Kuniya.
\newblock A reaction-diffusion susceptible-vaccinated-infected-recovered model
  in a spatially heterogeneous environment with {D}irichlet boundary condition.
\newblock {\em Math. Comput. Simulation}, 190:848--865, 2021.
\newblock \href {https://doi.org/10.1016/j.matcom.2021.06.020}
  {\path{doi:10.1016/j.matcom.2021.06.020}}.

\bibitem{Webb}
G.~F. Webb.
\newblock {\em Population Models Structured by Age, Size, and Spatial
  Position}, pages 1--49.
\newblock Springer Berlin Heidelberg, Berlin, Heidelberg, 2008.
\newblock \href {https://doi.org/10.1007/978-3-540-78273-5_1}
  {\path{doi:10.1007/978-3-540-78273-5_1}}.

\bibitem{YusufBenyah}
T.~T. Yusuf and F.~Benyah.
\newblock Optimal control of vaccination and treatment for an sir
  epidemiological model.
\newblock {\em World journal of modelling and simulation}, 8(3):194--204, 2012.

\bibitem{ZhouXiangLi}
M.~Zhou, H.~Xiang, and Z.~Li.
\newblock Optimal control strategies for a reaction–diffusion epidemic
  system.
\newblock {\em Nonlinear Analysis: Real World Applications}, 46:446--464, 2019.
\newblock URL:
  \url{https://www.sciencedirect.com/science/article/pii/S146812181830542X},
  \href {https://doi.org/10.1016/j.nonrwa.2018.09.023}
  {\path{doi:10.1016/j.nonrwa.2018.09.023}}.

\end{thebibliography}
\end{document}